\documentclass[11pt,oneside]{article}

\usepackage[ruled]{algorithm2e}
\usepackage{bm,enumerate}
\usepackage{amssymb,amsfonts,amsthm}
\usepackage{amsmath}
\allowdisplaybreaks
\usepackage{mathtools,xfrac}
\usepackage{url,color}
\usepackage{graphicx}
\usepackage{enumitem}
\usepackage[sort,numbers]{natbib}

\usepackage{verbatim}
\usepackage[colorlinks=true,breaklinks=true,bookmarks=true,urlcolor=blue,citecolor=blue,linkcolor=blue,bookmarksopen=false,draft=false]{hyperref}

\usepackage{tikz}
\usepackage{pgfplots}
\pgfplotsset{compat=1.3}

\usepackage[margin=1in]{geometry}
\usepackage{authblk}
\makeatletter
\def\imod#1{\allowbreak\mkern10mu({\operator@font mod}\,\,#1)}
\makeatother
\usepackage[compact]{titlesec}

\theoremstyle{plain}
\newtheorem{theorem}{Theorem}[section]

\newtheorem{lemma}[theorem]{Lemma}
\newtheorem*{lemma*}{Lemma}
\newtheorem{proposition}[theorem]{Proposition}
\newtheorem*{proposition*}{Proposition}

\newtheorem*{observation*}{Observation}

\theoremstyle{definition}

\newtheorem{assumption}{Assumption}[section]

\theoremstyle{remark}
\newtheorem{remark}{Remark}[section]
\newtheorem*{remark*}{Remark}
\newtheorem{example}[theorem]{Example}
\newtheorem*{example*}{Example}

\DeclareMathOperator{\id}{id}

\definecolor{gold}{rgb}{0.85,0.65,0}

\newcommand{\ip}[2]{\left\langle #1 , #2 \right\rangle}    %

\def\beq{\begin{equation}}
\def\eeq{\end{equation}}

\def\fnote#1{\footnote}
\newcommand{\epr}{\hfill\hbox{\hskip 4pt \vrule width 5pt height 6pt depth 1.5pt}\vspace{0.0cm}\par}

\def\ra{\rangle}
\def\la{\langle}

\newcommand{\grad}{\ensuremath{\nabla}}

\def\N{{\mathbb{N}}}

\def\R{{\mathbb{R}}}

\def\cO{{\cal O}}
\def\cP{{\cal P}}

\newcommand{\bbE}{\mathbb{E}}

\newcommand{\bbN}{\mathbb{N}}

\newcommand{\bbP}{\mathbb{P}}

\newcommand{\bbR}{\mathbb{R}}

\newcommand{\CB}{\mathcal{B}}

\newcommand{\CF}{\mathcal{F}}

\newcommand{\CP}{\mathcal{P}}

\DeclareMathOperator{\Proj}{Proj}

\DeclareMathOperator{\argmin}{arg\,min}

\DeclareMathOperator{\Skew}{Skew}

\DeclareMathOperator{\Tr}{Tr}

\DeclareMathOperator{\Image}{Im}

\DeclareMathOperator{\Exp}{Exp}
\DeclareMathOperator{\St}{St}
\DeclareMathOperator{\Expm}{Expm}

\def\dim{\mathop{{\rm dim}\,}}

\DeclareMathOperator{\Span}{Span}

\DeclareMathOperator{\Ker}{Ker}

\def\log{\mathop{{\rm log}}}

\def\Halmos{}

\begin{document}

\title{Coordinate Descent Without Coordinates:\\ Tangent Subspace Descent on Riemannian Manifolds}
\author[1]{David Huckleberry Gutman}
\author[2]{Nam Ho-Nguyen}
\affil[1]{Texas Tech University}
\affil[2]{University of Sydney}

\maketitle

\begin{abstract}
We extend coordinate descent to manifold domains, and provide convergence analyses for geodesically convex and non-convex smooth objective functions. Our key insight is to draw an analogy between coordinate blocks in Euclidean space and tangent subspaces of a manifold. Hence, our method is called tangent subspace descent (TSD). The core principle behind ensuring convergence of TSD is the appropriate choice of subspace at each iteration. To this end, we propose two novel conditions, the gap ensuring and $C$-randomized norm conditions on deterministic and randomized modes of subspace selection respectively, that promise convergence for smooth functions and that are satisfied in practical contexts. We propose two subspace selection rules of particular practical interest that satisfy these conditions: a deterministic one for the manifold of square orthogonal matrices, and a randomized one for the Stiefel manifold. Our proof-of-concept numerical experiments on the orthogonal Procrustes problem demonstrate TSD's efficacy.

\end{abstract}

\section{Introduction}\label{sec:intro}

The block coordinate descent (BCD) method \cite{BeckTetruashvili2013,Nesterov2012,RichtarikTakac2014,FrongilloReid2015,Leventhal10} is a popular choice of algorithm for solving
\[
\min_{x\in \R^n}f(x),
\]
where $f:\R^n \to \R$ is a continuously differentiable function. At each iteration, BCD employs a single gradient step to $f$ in a block of variables, chosen at random or in a cyclic fashion, while holding all other variables constant. Symbolically, we can write this key update step as
\begin{equation*}\label{eqn:BCD}
y^k = y^{k-1} - P_{i_k} \nabla f(y^{k-1}), \quad k=1,2,\ldots
\end{equation*}
where $\{P_1,\ldots, P_m\}$ is a set of matrices each with $n$ rows such that $\begin{bmatrix} P_1 & | & \hdots & | & P_m\end{bmatrix}=I_{n\times n}$, and $i_k\in\{1,\ldots,m\}$ is selected in a random or cyclic fashion. One may regard this as an alternating minimization scheme where exact minimization with respect to a block of variables is replaced with a gradient descent step. When $n$ is extremely large, BCD is advantageous compared to standard gradient-based algorithms, since computing $P_{i_k} \nabla f(y^{k-1})$ only requires computing partial derivatives in the $P_{i_k}$ block, and updating only these variables, is relatively easy compared with computing the full gradient $\grad f(y^{k-1})$ and updating all variables.

Due to the increasing importance of optimization on manifolds to large-scale data science problems, there has been a flurry of activity analyzing first-order methods adapted to this problem class \cite{HuangWei2019,ChenEtAl2018,LiuBoumal2019,ZhangEtAl2019,ZhangEtAl2016,ZhangSra2018,LiuEtAl2017,SatoEtAl2019,FerreiraEtAl2019}. In contrast to many other methods depending on first-order information, such as gradient descent and stochastic gradient descent, it is not obvious how to extend BCD to capture problems where the feasible set is an arbitrary Riemannian manifold $M$. The primary challenge in adapting BCD to an arbitrary manifold is its inherent dependence on the linear structure of $\R^n$, which allows us to reuse the same block coordinate decomposition $\{P_1,\ldots,P_m\}$ at different iterations.

This paper extends BCD methods to the problem
\[
\min_{x \in M} f(x)
\]
where $M$ is an $n$-dimensional Riemannian manifold and $f:M\to\R$ is a continuously differentiable function. Our key insight is that inexact alternating minimization over blocks of coordinates in $\R^n$ generalizes to inexact alternating minimization over subspaces of the manifold's tangent space at the current iterate. The resultant algorithm, \emph{Tangent Subspace Descent} (TSD), is the focus of our paper. TSD reformulates the key update step of BCD as
\begin{equation*}\label{eqn:TSD}
y^k = \Exp_{y^{k-1}}\left(-\eta \cP(y^{k-1},k)\nabla f(y^{k-1})\right), \quad k=1,2,\ldots
\end{equation*}
where $\Exp_{y^k}$ is the manifold's exponential map at $y^k$, $\nabla f$ denotes the Riemannian gradient of $f$, $\cP(y^{k-1},k)$ is an orthogonal (with respect to the Riemannian metric) projection onto a subspace of the tangent space $T_{y^{k-1}} M$ at $y^{k-1}$. We formally define these diffeo-geometric objects in Section \ref{sec:problemdef}.

TSD is a generalization of BCD, since a block of coordinates can also be interpreted as a tangent subspace. In the Euclidean setting $M=\bbR^n$, tangent spaces at diffferent points are isomorphic to $\bbR^n$ itself, and are easily identified with each other. This is the key property that allows us to reuse the same block coordinate decomposition $\{P_1,\ldots,P_m\}$ at different iterations. By using the same decomposition, BCD ensures that all possible directions of descent lay in the span of the projections, and that the collection of possible descent directions do not progressively cluster near a lower dimensional subspace. Such clustering can destroy the convergence of an iterative method if the gradient remains roughly orthogonal to said lower dimensional subspace. Examples of this phenomenon are presented in Theorems \ref{thm:projection-failure-randomized} and \ref{thm:projection-failure}.

\subsection{Contributions and Related Work}\label{subsec:contribution}

Different tangent spaces on a general manifold $M$, however, are not so easily related as in the Euclidean case, which necessitates the introduction of $\cP$ in TSD, what we call a \emph{subspace selection rule}. Unlike the Euclidean setting, the subspace selection rule $\cP$ can choose subspace decompositions that may vary across iterations $t$, and may be non-orthogonal. A key goal of our paper is to precisely define new concepts which quantify how $\cP$ affects the span and clustering of descent directions on a general manifold, and show how these can be used to analyze the convergence of TSD. Our key contributions are as follows.
\begin{itemize}
\item Through counterexamples, we demonstrate that care must be taken in specifying $\CP$. Section \ref{sec:general-selection} describes subspace selection rules which satisfy simple conditions one might believe ensure convergence, but in fact, do not. This motivates the introduction of the novel and mild conditions, the $(\gamma,r)$-\emph{gap ensuring} (Assumption \ref{ass:gap}) and the $C$-\emph{randomized norm} (Assumption \ref{ass:random-norm}) conditions for deterministic and randomized rules $\cP$ respectively. These conditions quantify the `information loss' of $\cP$.

\item Our analysis, namely Theorems \ref{thm:deterministic-descent-rates} and \ref{thm:randomized-descent-rates}, prove that these assumptions ensure convergence of TSD. The convergence analysis for deterministic rules with Assumption \ref{ass:gap} requires a non-trivial modification of the existing analysis for BCD (see Remark \ref{rem:gap-ensuring-distance}). The convergence analysis for randomized rules is sufficiently general to allow selection from an infinite collection of subspaces, unlike in previous work.

\item Our new conditions restrict the choice of $\cP$, but are still sufficiently flexible to capture a number of useful settings. We provide examples of different manifolds and subspace selection rules in Sections \ref{sec:deterministic-examples} and \ref{sec:randomized-examples}. In particular, Section \ref{sec:randomized-stiefel} presents a new randomized selection rule for Stiefel manifolds, which satisfies the $C$-randomized norm condition and thus ensures convergence of TSD. In addition, Section \ref{sec:orthogonal} elaborates on the special case of the manifold of orthogonal $n \times n$ real matrices, for which we provide a class of deterministic subspace selection rules, and show that it satisfies the $(\gamma,r)$-gap ensuring condition for convergence of TSD. Furthermore, we show that this class of rules result in computationally cheap and sparse updates for TSD, of order $O(n)$ cost per update as opposed to $O(n^3)$ for standard Riemannian gradient descent updates.

\item We conduct proof-of-concept numerical experiments on the  well-known orthogonal Procrustes problem $\min_{X \in O_n} \|AX - B\|_F^2$, with $A\in\R^{d \times n}$, $B\in\R^{d \times n}$, and $O_n$ being the manifold of square orthogonal matrices, which demonstrate the efficacy of TSD over standard Riemannian gradient descent schemes.
\end{itemize}

\citet{FrongilloReid2015} considered randomized and not necessarily orthogonal subspace descent in the Euclidean setting. We significantly generalize this to cyclic and randomized subspace descent on manifolds. To our knowledge, the only papers which have considered coordinate descent in a non-Euclidean setting are \citet{ShalitChechik2014,GaoEtAl2018}. Both of these examine specifically the Stiefel manifold of orthogonal matrices. The algorithm by \citet{GaoEtAl2018} updates a single column of the matrix at each step, but then adds a so-called ``correction'' which involves computing a singular value decomposition, and is not a direct generalization of BCD. The algorithm by \citet{ShalitChechik2014} is a specialization of our proposed one to $O_n$, for a particular randomized subspace selection rule.

In contrast to \cite{FrongilloReid2015,GaoEtAl2018,ShalitChechik2014}, our work examines a more general setting on any manifold, with any geometry, for both deterministic and randomized (valid) subspace selection rules. Our analysis also gives global convergence guarantees for geodesically convex objectives, in addition to convergence to stationary points for non-convex objectives. Of particular note, the randomized selection rule in Theorem \ref{thm:stiefel-selection} for general Stiefel manifolds broadly generalizes the randomized rule in \cite{ShalitChechik2014} for $O_n$. Furthermore, to our knowledge, the deterministic selection rule in Section \ref{sec:orthogonal} is the first for optimization over $O_n$.

Note also that our gap ensuring condition is closely related to principal angles between subspaces, which are known to control the convergence rates of cyclic projection schemes as far back as \cite{Deutsch95}. Principal angles between subspaces have also appeared in convergence rates for alternating projection schemes on manifolds \cite{Lewis08}.
	
\subsection{Outline}

In Section \ref{sec:problemdef}, we provide the language of, and convergence analysis tools for, manifold-constrained optimization. In Section \ref{sec:general}, we present Tangent Subspace Descent and two theorems (Theorems \ref{thm:projection-failure} and \ref{thm:projection-failure-randomized}) explaining the need for the judicious selection of subspace selection rules. In Section \ref{sec:deterministic}, we elaborate the $(\gamma,r)$-gap ensuring condition, delineate examples of deterministic subspace selection rules satisfying it (including cyclic BCD), and analyze the convergence TSD equipped with such a rule for smooth (both g-convex and non g-convex) functions. In Section \ref{sec:randomized}, we repeat this program for randomized subspace selection rules $C$-randomized norm. We show that randomized BCD can be seen as an instance of TSD with a rule satisfying this condition. This section also provides the new randomized subspace selection rule for general Stiefel manifolds (Theorem \ref{thm:stiefel-selection}). In Section \ref{sec:orthogonal}, we propose the first deterministic rule for subspace selection over the set of square orthogonal matrices and apply TSD to the orthogonal Procrustes problem.

\section{Optimization on Manifolds}\label{sec:problemdef}

In this section, we present some of the basic terminology and convergence analysis tools for optimization on manifolds. We will regularly use the standard notation and known results of Riemannian manifolds, cf. \citet{Lee2012book,Lee2018book,AbsilEtAl2007book,Vishnoi2018}. We consider the following general optimization problem
\begin{equation}\label{eqn:opt-problem}
f^* = \min_x \left\{ f(x) : x \in M \right\},
\end{equation}
where $M$ is an $n$-dimensional Riemannian manifold, and $f:M\to\R$ is a differentiable function. 
 
For each point $x\in M$, we denote the \emph{tangent space} to $M$ at $x$ as $T_x M$. Recall that $T_x M$ is an $n$-dimensional real vector space. We denote the tangent bundle of $M$, the union of all tangent spaces of $M$, as $TM:=\bigcup_{x\in M} T_x M$; it is a smooth manifold in its own right. A \emph{vector field} is a smooth map $V:M \to TM$ satisfying $V(x) \in T_x M$ for each $x \in M$. The \emph{Riemannian metric} on $M$ is a collection of inner products $\{\la \cdot, \cdot \ra_x\}_{x\in M}$ on each $T_x M$ which is smoothly varying in the sense that for any two vector fields $V,W:M \to TM$, $x \mapsto \la V(x), W(x) \ra_x$ is a smooth function on $M$. The norm on $T_xM$ induced by $\ip{\cdot}{\cdot}_x$ is denoted $\|\cdot\|_x$. 

The \emph{length} of a piecewise smooth curve $\gamma:[0,1]\to M$ with $\gamma(0)=x$, $\gamma(1)=y$, and derivative $\gamma':[0,1]\to TM$ is defined as
\[ 
L(\gamma)=\int_0^1 \|\gamma'(t)\|_{\gamma(t)} dt.
\]
The \emph{Riemannian distance} $d(x,y)$ between points $x$ and $y$ is the infimum of the lengths of all piecewise smooth curves joining $x$ and $y$. The manifold $M$ is a metric space when endowed with the Riemannian distance. If $M$ is complete and connected under this metric then the distance can be attained by some smooth curve. We say that a curve $\gamma:[0,1] \to M$ with $\gamma(0)=x$ and $\gamma(1)=y$ is a \emph{geodesic between $x$ and $y$} if it is locally distance minimizing. Given any two points $x,y\in M$, we can find a geodesic $\gamma$ between them such that $L(\gamma) = d(x,y)$. 

The \emph{exponential map} at a point $x \in M$ is denoted $\Exp_x:T_x M \to M$ gives us a way to generalize the operation $x+v$ in Euclidean space, where $v \in T_x M$. We have $\Exp_x(0) = x$ for any $x \in M$. The exponential map gives rise to geodesics: the curve $t \mapsto \gamma(t) := \Exp_x(tv)$ is always a geodesic between $x$ and $\Exp_x(v)$. We assume that $\Exp_x(v)$ is well defined for all $x \in M$, $v \in T_x M$, that is, $M$ is \emph{geodesically complete}. The Hopf-Rinow theorem states that geodesic completeness and completeness under the Riemannian distance are equivalent. Thus, this is a decidedly mild assumption. We define a set-valued \emph{inverse exponential map}, $\Exp_x^{-1}: M \rightrightarrows T_x M$ as \[
\Exp_x^{-1}(y) = \argmin_{v \in T_x M} \left\{ \|v\|_x : \Exp_x(v) = y \right\}.
\]
Furthermore, we have that for any $v \in \Exp_x^{-1}(y)$, $d(x,y) = \|v\|_x$, and for general $v \in T_x M$, we have $\|v\|_x \geq d(x,\Exp_x(v))$.

For any smooth curve $\gamma:I \to M$ defined on an interval $I \subset \bbR$, there is a well-defined collection of invertible linear maps called \emph{parallel transport operators along $\gamma$}, denoted $\Gamma_{\gamma,t^0}^{t^1}: T_{\gamma(t^0)} M \to T_{\gamma(t^1)} M$, for $t^0,t^1 \in I$. These maps satisfy $(\Gamma_{\gamma,t^0}^{t^1})^{-1} = \Gamma_{\gamma,t^1}^{t^0}$ and $\Gamma_{\gamma,t^0}^{t^0} = \id_{T_{\gamma(t^0)} M}$. Furthermore, a key property is that they preserve the metric, i.e., $\la v,v' \ra_{\gamma(t^0)} = \la \Gamma_{\gamma,t^0}^{t^1} v, \Gamma_{\gamma,t^0}^{t^1} v' \ra_{\gamma(t^1)}$ for any $t^0,t^1 \in I$, $v,v' \in T_{\gamma(t^0)} M$. In this paper, we will only be interested in parallel transport operators along geodesic curves. When $y=\Exp_x(v)$, we use the shorthand $\Gamma_x^y: T_x M \to T_y M$ to denote $\Gamma_{\gamma,0}^1$, where $\gamma(t) := \Exp_x(tv)$ is a geodesic between $x$ and $y$. Note that while $\Gamma_x^y$ depends on $v$, we will suppress this in the notation, since the dependence will be implied from $y = \Exp_x(v)$, which will be clear from context. We will also use the notation $\Gamma_x^y$ when there is a \emph{unique} geodesic curve between $x$ and $y$, i.e., when $\Exp_x^{-1}(y) = \{v\}$ is a singleton. When $y = \Exp_x(v)$ is not clear from context, the fact that $\Exp_x^{-1}(y) = \{v\}$ will be made clear. Note that if $y = \Exp_x(v)$, then $x = \Exp_y(-\Gamma_x^y v)$ and $(\Gamma_x^y)^{-1} = \Gamma_y^x$.

For a differentiable function $f:M\to\R$, the \emph{Riemannian gradient} of $f$ is defined as the element $\nabla f(x)\in T_x M$ such that $\ip{\nabla f(x)}{v}_x=df_x v$ for all $v\in T_x M$, where $df_x: T_x M\to\R$ is the differential of $f$ at $x\in M$. The gradient is key to the analysis of optimization algorithms on manifolds as the necessary first-order stationarity condition $\grad f(x^*) = 0$ for minimizers of \eqref{eqn:opt-problem} still holds. In this case, we say that $x^*$ is a \emph{stationary point} of $f$. This condition is also sufficient when $f$ is $g$-convex: a condition defined shortly. When $f$ is continuously differentiable, it holds that if $x_n\to x^*$ and $\|\nabla f(x_t)\|_{x_t}\to 0$ then $\nabla f(x^*)=0$.

We are interested in two classes of functions $f:M\to\R$ on geodesically complete Riemannian manifolds defined by two different conditions on the gradient:
\begin{align}
&\forall (x,v) \in TM, y=\Exp_x(v), \ \|\Gamma_x^y \grad f(x) - \grad f(y)\|_y \leq L_f \, d(x,y). \label{eqn:smooth}\\
&\forall x,y \in M, v \in \Exp_x^{-1}(y), \ f(y) \geq f(x) + \la \grad f(x), v \ra_x.\label{eqn:g-convex}
\end{align}
We say that a function $f$ satisfying  \eqref{eqn:smooth} is \emph{$L_f$-smooth}, and a function satisfying \eqref{eqn:g-convex} is \emph{geodesically convex}, or g-convex. Note that $\|\Gamma_x^y \grad f(x) - \grad f(y)\|_y = \| \grad f(x) - \Gamma_y^x \grad f(y)\|_x$, and that $L_f$-smoothness of $f$ implies the following upper bound:
\begin{equation}\label{eqn:smooth-ub}
\forall x \in M, v \in T_x M, \ f(\Exp_x(v)) \leq f(x) + \la \grad f(x), v \ra_x + \frac{L_f}{2} \|v\|_x^2.
\end{equation}
In this paper, we use the term smoothness of the function to mean both infinite differentiability and Lipschitz continuity of the gradient. Which definition we mean will be clear from context. Furthermore, \citet[Theorem 2.3]{FerreiraOliveira1998} show that g-convexity is equivalent to the following: for any $x,y \in M$, $v \in \Exp_x^{-1}(y)$, and $\beta \in [0,1]$, $f(\Exp_x(\beta v)) \leq (1-\beta) f(x) + \beta f(y)$.

The convergence analyses of (non-accelerated) gradient-based methods for\\ smooth convex optimization problems is often based on establishing the following so-called `sufficient decrease' relation between iterates:
\begin{equation}\label{eqn:sufficient-decrease}
f(x^{t-1}) - f(x^t) \geq \eta \|\grad f(x^{t-1})\|_{x^{t-1}}^2.\tag{Dec}
\end{equation}
Often, \eqref{eqn:sufficient-decrease} is derived only from smoothness of $f$ rather than convexity: if $f$ is $L_f$-smooth then by selecting $v=-\frac{1}{L_f}\nabla f(x)$ in \eqref{eqn:smooth-ub},
\begin{equation}\label{eqn:smooth-ub-minimum}
f(x)-f\left(\Exp_x\left(-\frac{1}{L_f}\nabla f(x)\right)\right)\geq \frac{1}{L_f}\left\|\nabla f(x)\right\|_x^2.
\end{equation}
Thus, when $f$ only satisfies smoothness, we can get a stationarity guarantee when this sufficient decrease condition is met. The analyses in Sections \ref{sec:deterministic} and \ref{sec:randomized} are based on providing conditions for which \eqref{eqn:sufficient-decrease}, or a suitable modification thereof, is satisfied.

\section{Tangent Subspace Descent on General Riemannian Manifolds}\label{sec:general}

In this section, we introduce the template for tangent subspace descent and elaborate the issues that arise when choosing an appropriate subspace selection rule $\cP$. With the key TSD step \eqref{eqn:TSD} in hand, we state the complete algorithmic template in Algorithm \ref{alg:tangent-subspace-descent}.

\begin{algorithm}[h]\caption{Tangent Subspace Descent}\label{alg:tangent-subspace-descent}
	\KwData{Initial point $x^0 \in M$, subspace selection rule $\cP$, $m\in\N$}
	\KwResult{Sequence $\{x^t\}_{t \geq 1} \subset M$.}
	\For{$t=1,2,\ldots$}{
		Set $y^{t,0} := x^{t-1}$\;	
		\For{$k \in [m]$}{		
			Pick $P_k^{y^{t,k-1}}:=\cP\left(\{y^{t,j}\}_{j=0}^{k-1},k \right)$ and $\eta_{t,k} >0$\;
			Update $y^{t,k}:=\Exp_{y^{t,k-1}}\left(-\eta_{t,k} P_k^{y^{t,k-1}} \nabla f(y^{t,k-1})\right)$\;
			}
		Update $x^t:=y^{t,m}$\;
		}
\end{algorithm}

At iteration $t \geq 1$ of Algorithm \ref{alg:tangent-subspace-descent}, given iterates $\left\{y^{t,k-1}\right\}_{k \in [m]}$ and projections $\left\{P_k^{y^{t,k-1}}\right\}_{k \in [m]}$, we define
\[ P_{t,k} := \Gamma_{y^{t,k-1}}^{y^{t,0}} P^{y^{t,k-1}}_k\Gamma_{y^{t,0}}^{y^{t,k-1}}, \ k\in[m], \qquad \|\cdot\|_{y^{t,0},\cP} := \sqrt{ \sum_{k \in [m]} \|P_{t,k} (\cdot) \|_{y^{t,0}}^2 }. \]
Note that $\|\cdot\|_{y^{t,0},\cP}$ is, in general, a seminorm on $T_{y^{t,0}} M$, and is a norm if and only if $\Span\left( \bigcup_{k \in [m]} \Image(P_{t,k}) \right) = T_{y^{t,0}} M$. Furthermore, if $P_{t,k}$ are projections onto orthogonal subspaces which span $T_{y^{t,0}} M$, then $\|\cdot\|_{y^{t,0},\cP} = \|\cdot\|_{y^{t,0}}$. This is the case for BCD, and leads to sufficient decrease bounds like \eqref{eqn:sufficient-decrease}, yet is not true in general for TSD. Instead, for TSD, further conditions must be imposed and a new analysis is required to relate the two norms. The choice of the subspace selection rule $\cP$ is a critical element of this. Note that $\cP$ can cover both deterministic and randomized rules. For randomized rules, we will always take $m=1$, and denote for convenience $P_t := \CP(y^{t,0},1) = P_1^{y^{t,0}}$ (this is a randomly chosen projection on $T_{x^{t-1}} M$). We spend the remainder of this section discussing issues regarding $\cP$ that can inhibit the convergence of Algorithm \ref{alg:tangent-subspace-descent} and that motivate the conditions we place upon $\cP$ in Sections \ref{sec:deterministic} and \ref{sec:randomized}.

\subsection{The Problem of Subspace Selection}\label{sec:general-selection}

Since $M$ is not (in general) a vector space, it cannot be decomposed into a collection of subspaces, so we need to pick different subspaces to project onto at each point and iteration. While the tangent space $T_x M$ is a vector space for each $x\in M$, the spaces $T_x M$ and $T_y M$ can be very different for $y\neq x$. Consider the example of the tangent space to the north pole of the sphere, $\mathbb{S}^2$. This space is nothing more than an affine plane parallel to the $xy$-plane that lies flat on the sphere at the pole. However, if we translate this plane to any point on the equator then we see the plane cuts through the sphere so it no longer lies tangent to the sphere. The differing structures of the tangent spaces at different points on $M$ thus disallows us from picking a fixed decomposition of a single tangent space that we use at every subsequent iteration, as in BCD.

Furthermore, the choice of subspaces dictates not only the convergence rate of this algorithm, but whether it converges at all. Even simple examples on $\R^n$ show how badly the algorithm can breakdown when $\cP$ is chosen poorly. When $V$ is a Hilbert space and $S\subseteq V$ is a subspace, we will let $\Pi_S$ denote the orthogonal projection onto $S$.

\begin{example}
Let $M=\R^n$, $\cP \equiv \Pi_{\Span(\{e_1\})}$, $m=1$, and $f(x)=\frac{1}{2}\|x\|_2^2$. If $x^0=e_1+e_2$ and $\eta_{t,k}=1$ then $x^t=e_2$ for $t\geq 1$, but $\argmin_{x\in\R^n}f(x)=\{0\}$.
\epr
\end{example}

\noindent The primary issue in the above example is that $\cP$ produces a spartan collection of subspaces. Specifically, $\Span(\{e_1\}) \neq\R^2$. Recasting this in terms of tangent spaces, $\Span\left(\Image \Gamma_y^x \Pi_{\Span(\{e_1\})}\Gamma_x^y\right)\neq T_x\R^2\cong\R^2$ for any $x,y\in\R^2$. 

One is tempted to believe for deterministic rules it might suffice simply to pick a $\cP$ that promises a ``spanning condition" such as
\begin{equation}\label{eqn:spanning-condition}
\Span\left(\bigcup_{k \in [m]} \Image(P_{t,k})\right)=T_{y^{t,0}} M
\end{equation}
for any sequence $y^{t,0},\ldots,y^{t,m}$ generated by one execution of the inner loop in Algorithm \ref{alg:tangent-subspace-descent}. Note that this spanning condition is readily seen to be equivalent to the requirement that $\|\cdot\|_{y^{t,0},\cP}$ is a norm (as opposed to just a seminorm). Thus, for randomized rules, we are similarly tempted to conjecture that if
\begin{equation}\label{eqn:spanning-condition-random}
v \mapsto \sqrt{\bbE[\|P_t v\|_{x^{t-1}}^2 \mid x^{t-1}]}
\end{equation}
is a norm on $T_{x^{t-1}} M$ then convergence is guaranteed. Unfortunately, these conditions are insufficient. We now exhibit deterministic and randomized selection rules $\cP$ satisfying \eqref{eqn:spanning-condition} and \eqref{eqn:spanning-condition-random} for which convergence of TSD fails.

For both examples, we take $M = \bbR^n$ and $f(x) = \frac{1}{2} \|x\|_2^2$, which is a $1$-smooth function. Note also that $T_x M = \bbR^n$ for each $x \in \bbR^n$, $\Exp_x(v) = x+v$, and the Riemannian metric $\la \cdot, \cdot \ra_x$ is just the usual inner product on $\bbR^n$. For the deterministic setting, we take $m = n$ in Algorithm \ref{alg:tangent-subspace-descent}, and denote for convenience $\|\cdot\|_{t,\CP} := \|\cdot\|_{y^{t,0},\CP}$. As mentioned above, for the randomized setting, we have $m=1$ and $P_t := \cP(y^{t,0},1)$ is a random projection on $T_{x^{t-1}} M$. Furthermore, the stepsizes in Algorithm \ref{alg:tangent-subspace-descent} will be taken as $1$, since $\grad f(x) = x$ and for any orthogonal projection $\Pi:\bbR^n \to \bbR^n$, we have
\[ \min_\eta f(x-\eta \Pi \grad f(x)) = \frac{1}{2} \|x - \Pi x\|_2^2 = f(x) - f(\Pi x). \]

We start with the randomized setting.
\begin{theorem}\label{thm:projection-failure-randomized}
There exists a randomized projection selection rule $\CP$ satisfying \eqref{eqn:spanning-condition-random} and constant $\epsilon > 0$ such that, at any $t \geq 1$ in Algorithm \ref{alg:tangent-subspace-descent}, $f(x^t) > \epsilon$. Furthermore, \[
\lim_{t\to\infty} \sup_{v \neq 0} \frac{\|v\|_2}{\sqrt{\bbE[\|P_t v\|_2^2 \mid x^{t-1}]}} = \infty
\]
almost surely.
\end{theorem}
\begin{proof}[Proof of Theorem \ref{thm:projection-failure-randomized}]
We define the following projection selection rule. Given a starting point $x^0 \neq 0$, denote $\epsilon := f(x^0)/2 > 0$. For $t \geq 1$, let $H_t = \left\{ v \in \bbR^n : \la v,x^{t-1} \ra = \sqrt{f(x^{t-1}) - \epsilon} \right\} \cap \left\{ v \in \bbR^n : \|v\|_2 = 1\right\}$. We choose $\{v_{t,i}\}_{i \in [n]} \subset H_t$ such that $\Span\left( \{v_{t,i}\}_{i \in [n]} \right) = \bbR^n$ (this is possible due to a simple dimension counting argument). We then choose $v_t$ uniformly at random from $\{v_{t,i}\}_{i \in [n]}$, and set $P_t$ to be the projection onto the subspace spanned by $v_t$, i.e., $x \mapsto \la v_t,x \ra v_t$.

We thus have $x^t = x^{t-1} - P_t x^{t-1}$ and
\[ f(x^t) = f(x^{t-1}) - f(P_t x^{t-1}) = f(x^{t-1}) - \frac{1}{2} \la v_t, x^{t-1} \ra^2 = \frac{f(x^{t-1}) + \epsilon}{2}. \]
Solving this recurrence relation with the initial condition $f(x^0) = 2\epsilon$, we get $f(x^t) = \epsilon(1+2^{-t}) > \epsilon$.

Note that this also implies that for any $t \geq 1$, $\la v_t, x^{t-1} \ra = \sqrt{f(x^{t-1}) - \epsilon} = \sqrt{\epsilon 2^{-t}}$. Consequently, we see that
\[ x^t = x^0 + \sum_{s \in [t]} \la v_s, x^{s-1} \ra v_s = x^0 + \sqrt{\epsilon} \sum_{s \in [t]} 2^{-s/2} v_s \]
must converge to some $\bar{x} \in \bbR^n$, since $\{2^{-s/2}\}_{s \geq 1}$ is an absolutely convergent sequence and $\{v_s\}_{s \geq 1}$ are from the unit ball. It follows that for any $t \geq 1$ and $i \in [n]$, $\la v_{t,i}, \bar{x} \ra \to 0$. Thus,
\[ \bbE\left[ \|P_t \bar{x}\|_2^2 \mid x^{t-1} \right] = \frac{1}{n} \sum_{i \in [n]} \la v_{t,i}, \bar{x} \ra^2 \to 0, \]
so
\[ \lim_{t\to\infty} \sup_{v \neq 0} \frac{\|v\|_2}{\sqrt{\bbE[\|P_t v\|_2^2 \mid x^{t-1}]}} \geq \lim_{t\to\infty} \frac{\|\bar{x}\|_2}{\sqrt{\bbE[\|P_t \bar{x}\|_2^2 \mid x^{t-1}]}} = \infty. \]
\Halmos
\end{proof}

A non-convergent deterministic rule can be constructed in a similar manner, with small modifications.

\begin{theorem}\label{thm:projection-failure}
There exists a deterministic projection selection rule $\CP$ satisfying \eqref{eqn:spanning-condition} with $m=n$ and constant $\epsilon > 0$ such that, at any $t \geq 1$ in Algorithm \ref{alg:tangent-subspace-descent}, $f(x^t) > \epsilon$. Furthermore, 
\[
\lim_{t\to\infty} \sup_{v \neq 0} \frac{\|v\|_2}{\|v\|_{t,\cP} } = \infty. 
\]
\end{theorem}
\begin{proof}[Proof of Theorem \ref{thm:projection-failure}]
We define the following projection selection rule. Given a starting point $x^0 \neq 0$, denote $\epsilon := f(x^0)/2 > 0$. Fix $t \geq 1$, and set $y^{t,0} = x^{t-1}$, $S_{t,0} = \{0\}$. For $k \in [m]$, we do the following:
\begin{itemize}
	\item Let $H_{t,k} := \{ v : \la v,y^{t,k-1} \ra = \sqrt{(f(x^{t-1})-\epsilon)/m} \}$.
	\item Choose $v_{t,k} \in H_{t,k} \cap \{v : \|v\|_2=1\}$ such that $\dim(S_{t,k-1} \cup \{v_{t,k}\}) = \dim(S_{t,k-1})+1$, where $\dim(S)$ is the affine dimension of $S$ (this is possible due to a simple dimension counting argument).
	\item Set $S_{t,k} := S_{t,k-1} \cup \{v_{t,k}\}$.
	\item Define $\cP(y^{t,k-1},k)$ as the projection onto the subspace spanned by $v_{t,k}$.
\end{itemize}
Set $x^t = y^{t,n}$. Also observe that $\dim\left( \{v_{t,k}\}_{k \in [m]} \right) = m=n$ so $\Span\left( \{v_{t,k}\}_{k \in [m]} \right) = \bbR^n$. Then for each $k \in [m]$ we have
\begin{align*}
f(y^{t,k}) = f(y^{t,k-1}) - \frac{1}{2} \la v_{t,k}, y^{t,k-1} \ra^2=\hdots= f(y^{t,0}) - \frac{1}{2} \sum_{j \in [k]} \la v_{t,j}, y^{t,j-1} \ra^2&= f(x^{t-1}) - \frac{k}{2m} (f(x^{t-1}) - \epsilon)\\
&= \frac{(2m-k) f(x^{t-1}) + k \epsilon}{2m}.
\end{align*}
In particular, $f(x^t) = (f(x^{t-1})+\epsilon)/2$, so as in the proof of Theorem \ref{thm:projection-failure-randomized}, $f(x^t) = \epsilon(1+2^{-t}) > \epsilon$. Similar arguments to the proof of Theorem \ref{thm:projection-failure-randomized} show that $x^t \to \bar{x}$, $\la v_{t,k}, \bar{x} \ra \to 0$ for any $k \in [m]$, and $\frac{\|\bar{x}\|_2}{\|\bar{v}\|_{t,\cP} } \to \infty$ (recognize that $\|\bar{x}\|_{t,\cP} = \sqrt{\sum_{k \in [m]} \la v_{t,k}, \bar{x} \ra^2}$).
\Halmos
\end{proof}

The primary issue that prevents convergence in Theorems \ref{thm:projection-failure-randomized} and \ref{thm:projection-failure} is that, while the $\{v_{t,k}\}_{k \in [m]}$ span $\R^n$, they eventually cluster near a lower-dimensional subspace $\{v : \la v, \bar{x} \ra = 0\}$ as the algorithm progresses. As a result, the projections cause some information loss in the direction $\bar{x}$.

It seems reasonable to conjecture that if the projections maintain a certain level of `orthogonality' (i.e., do not cluster near some lower-dimensional subspace), then the degeneracy of Theorems \ref{thm:projection-failure-randomized} and \ref{thm:projection-failure} cannot occur.  Constructing viable deterministic and randomized subspace selections rules that respect this observation is the thrust of the main two sections (Sections \ref{sec:deterministic} and \ref{sec:randomized}) of the paper. In the respective cases, the orthogonality of the selected subspaces will be reflected by the quantities $\sup_{v\in T_{y^{t,0}}M\setminus\{0\}}\frac{\|v\|_{y^{t,0}}}{\|v\|_{y^{t,0},\cP}}$ and $\sup_{v\in T_{x^{t-1}}M\setminus\{0\}}\frac{\|v\|_{x^{t-1}}}{\sqrt{\bbE\left[ \|P_t v\|_{x^{t-1}}^2 \mid x^{t-1} \right]}}$. As we will see, providing finite bounds on these quantities is sufficient to guarantee convergence of TSD, so the main challenge will be to provide conditions to guarantee this.

\section{Deterministic Subspace Selection}\label{sec:deterministic}

In this section, we analyze the convergence of TSD when it is equipped with a deterministic subspace selection rule $\cP$.
In Section \ref{sec:gap}, we motivate and present a novel and mild condition on $\cP$, the $(\gamma,r)$-gap ensuring condition, which guarantees convergence of tangent subspace descent for smooth functions. We then give a convergence analysis for $(\gamma,r)$-gap ensuring rules. In Section \ref{sec:deterministic-examples}, we present a series of important examples of selection rules $\cP$ that satisfy this condition including product manifolds, which includes block cyclic coordinate descent in Euclidean space as a special case, as well as an important non-Euclidean application, tensor PCA. Note that we also construct a practical $(\gamma,r)$-gap ensuring $\cP$ for the manifold of orthogonal square matrices, but we defer its construction to Section \ref{sec:orthogonal}.

\subsection{A Convergent Deterministic \texorpdfstring{$\cP$}{Lg} via the \texorpdfstring{$(\gamma,r)$}{Lg}-Gap Ensuring Condition}\label{sec:gap}

Given a Hilbert space $V$, we will say that a set of orthogonal projections $\{P_k'\}_{k \in [m]}$ is an \emph{orthogonal decomposition} of $V$ if $\Image(P'_i)\perp \Image(P_j)$ for $j\neq i$ and $\Span\left(\bigcup_{k \in [m]} \Image(P_k') \right) = V$. The foundation of our analysis is Assumption \ref{ass:gap} below, the $(\gamma,r)$-gap ensuring condition. We now describe why the norm equivalence terms from Theorems \ref{thm:projection-failure-randomized} and \ref{thm:projection-failure} arise in our analysis, as well as why Assumption \ref{ass:gap} is required for bounding them.

First, recall that we assume that our objective function $f$ is $L_f$-smooth, in the sense of \eqref{eqn:smooth-ub}. Similar to \eqref{eqn:smooth-ub}, smoothness will facilitate stepsize selection in order to achieve sufficient decrease as in \eqref{eqn:sufficient-decrease}. We refine the definition of smoothness slightly to individual subspaces $P_k^{y^{t,k-1}} T_{y^{t,k-1}} M$ in an inner loop.

\begin{assumption}[Generalized Block Smoothness]\label{ass:block-smooth}
	There exists constants $0 < L_1,\ldots,$ $L_m < \infty$ such that for any $t \geq 1$ and $k \in [m]$, the iterates of Algorithm \ref{alg:tangent-subspace-descent} satisfy
	\begin{equation}\label{eqn:smooth-block-generalized}
	f\left(\Exp_{y^{t,k-1}}(P_k^{y^{t,k-1}} v) \right) \leq f(y^{t,k-1}) + \ip{\grad f(y^{t,k-1})}{P_k^{y^{t,k-1}} v}_{y^{t,k-1}} + \frac{L_k}{2} \left\|P_k^{y^{t,k-1}} v\right\|_{y^{t,k-1}}^2
	\end{equation}
	for all $v\in T_{y^{t,k-1}}M$.
\end{assumption}
In practice, each $L_k$ may not be known, but can be implemented via backtracking line search (to achieve the condition \eqref{eq:deterministic-decrease-block}), which will terminate as long as $L_k$ is finite; this will hold since one can simply take $L_k = L_f$. From Assumption \ref{ass:block-smooth} and another mild condition, we obtain the following.

\begin{lemma}[Deterministic Sufficient Decrease Template]\label{lemma:sufficient-decrease-deterministic}
	Suppose Assumption \ref{ass:block-smooth} holds, and that at each $t \geq 1$, $k \in [m]$ in Algorithm \ref{alg:tangent-subspace-descent}, the stepsize is chosen as $\eta_{t,k} := 1/L_k$. Then
	\begin{equation}\label{eq:deterministic-decrease1}
	f(y^{t,0}) - f(y^{t,m}) \geq \sum_{k \in [m]} \frac{1}{2 L_k} \left\|P_k^{y^{t,k-1}} \grad f(y^{t,k-1}) \right\|_{y^{t,k-1}}^2. 
	\end{equation}
	Furthermore, suppose that there exists some $C>0$ such that for each $k \in [m]$, $\CP$ ensures the iterates of Algorithm \ref{alg:tangent-subspace-descent} satisfy
	\begin{equation}\label{eq:deterministic-decrease2}
	\left\|P_k \grad f(y^{t,0}) - P_k \Gamma^{y^{t,0}}_{y^{t,k-1}} \grad f(y^{t,k-1}) \right\|_{y^{t,0}}^2 \leq C \sum_{j \in [k]} \left\|P_j^{y^{t,j-1}} \grad f(y^{t,j-1})\right\|_{y^{t,j-1}}^2.
	\end{equation}
	Then for any $t \geq 1$,
	\begin{equation}\label{eq:deterministic-decrease}
	f(y^{t,0}) - f(y^{t,m}) \geq \frac{1}{4L_{\max}(1+Cm)}\left\|\grad f(y^{t,0})\right\|_{y^{t,0},\cP}^2.
	\end{equation}
\end{lemma}

\begin{proof}[Proof of Lemma \ref{lemma:sufficient-decrease-deterministic}]
For convenience, we fix a $t \geq 1$ and suppress it in the notation for the proof. Our stepsize selection, $\eta_{t,k}=\frac{1}{L_k}$, minimizes the right-hand side of \eqref{eqn:smooth-block-generalized}, which gives
\begin{equation}\label{eq:deterministic-decrease-block}
f(y^{k-1}) - f(y^k) \geq \frac{1}{2 L_k} \left\|P_k^{y^{k-1}} \grad f(y^{k-1})\right\|_{y^{k-1}}^2.
\end{equation}
We get \eqref{eq:deterministic-decrease1} by summing this over $k \in [m]$.
	
	For every $k \in [m]$,
	\begin{align*}
	&\left\| P_k \grad f(y^0)\right\|_{y^0}^2= \left\| P_k \grad f(y^0) - P_k \Gamma^{y^0}_{y^{k-1}} \grad f(y^{k-1}) + P_k \Gamma^{y^0}_{y^{k-1}} \grad f(y^{k-1})\right\|_{y^0}^2\\
	&\leq \left(\left\|P_k \grad f(y^0)-P_k \Gamma^{y^0}_{y^{k-1}} \grad f(y^{k-1}) \right\|_{y^0} + \left\| P_k \Gamma^{y^0}_{y^{k-1}} \grad f(y^{k-1}) \right\|_{y^0}\right)^2\\
	&\leq 2\left\|P_k \grad f(y^0)- P_k \Gamma^{y^0}_{y^{k-1}} \grad f(y^{k-1}) \right\|_{y^0}^2+2\|P_k \Gamma^{y^0}_{y^{k-1}} \grad f(y^{k-1}) \|_{y^0}^2\\
	&= 2\left\|P_k \grad f(y^0)-P_k \Gamma^{y^0}_{y^{k-1}} \grad f(y^{k-1}) \right\|_{y^0}^2+2\|\Gamma^{y^0}_{y^{k-1}} P^{y^{k-1}}_k \grad f(y^{k-1}) \|_{y^0}^2\\
	&= 2\left\|P_k \grad f(x)-P_k \Gamma^{y^0}_{y^{k-1}} \grad f(y^{k-1}) \right\|_{y^0}^2 + 2\| P^{y^{k-1}}_k \grad f(y^{k-1}) \|_{y^{k-1}}^2\\
	&\leq 2\left\| P^{y^{k-1}}_k \grad f(y^{k-1})\right\|_{y^{k-1}}^2 + 2 C \sum_{j \in [k]} \|P^{y^{j-1}}_j \grad f(y^{j-1})\|_{y^{j-1}}^2.
	\end{align*}
	where we apply the triangle inequality in line 2, the Cauchy-Schwarz inequality in line 3, the definition of $P_i$ in the rightmost term of line 4, and the fact that $\Gamma^{y^0}_{y^{k-1}}$ is an isometry in line 5.
	
	Summing this over $k \in [m]$ we have
	\begin{align*}
	\|\grad f(y^0)\|_{y^0,\cP}^2 = \sum_{k \in [m]} \left\|P_k \grad f(y^0)\right\|_{y^0}^2&\leq 2\sum_{k \in [m]} \left(1+(m-k)C\right)\left\|P^{y^{k-1}}_k \grad f(y^{k-1})\right\|_{y^{k-1}}^2\\
	&\leq 2\left(1+Cm\right)\sum_{k \in [m]} \left\|P^{y^{k-1}}_k \grad f(y^{k-1})\right\|_{y^{k-1}}^2.
	\end{align*}
	We then get \eqref{eq:deterministic-decrease} by simple rearrangement.\Halmos
\end{proof}

We will address the condition \eqref{eq:deterministic-decrease2} later. For now, we examine the result \eqref{eq:deterministic-decrease}. Since $x^{t-1} = y^{t,0}$, $x^t = y^{t,m}$, this is almost of the same form as \eqref{eqn:sufficient-decrease}, except the norm on the right hand side is $\|\cdot\|_{y^{t,0},\CP}$ instead of the required $\|\cdot\|_{y^{t,0}} = \|\cdot\|_{x^{t-1}}$. This also clarifies why the projection selection rule of Theorem \ref{thm:projection-failure} forces non-convergence of tangent subspace descent. In order to achieve the sufficient decrease condition \eqref{eqn:sufficient-decrease}, we need to replace $\|\nabla f(y^{t,0})\|_{y^{t,0},\cP}$ with $\|\nabla f(y^{t,0})\|_{y^{t,0}}$ which can be done if there exists a uniform (over $t \geq 1$) finite bound on the norm-equivalence term $\sup_{v \in T_{y^{t,0}} M \setminus \{0\}} \|v\|_{y^{t,0}}/\|v\|_{y^{t,0},\cP}$. For the example in Theorem \ref{thm:projection-failure}, it was shown that this quantity diverges to $\infty$ as $t \to \infty$.

As we will see in Example \ref{ex:parallel-transport-rule}, when $\CP$ is selected in a particular manner, this can be guaranteed. However, unless the manifold $M$ has special structure (e.g., product manifold, which includes Euclidean space) using this particular $\CP$ can be costly since it requires repeated computations of parallel transports. In order to circumvent this, we now introduce the $(\gamma,r)$-gap ensuring condition.
\begin{assumption}\label{ass:gap}
Let $r>0$ and $\gamma\in(0,1]$. Suppose that for any $t \geq 1$ and $\{y^{t,0},\ldots,y^{t,m}\}$ $\subset M$ generated by Algorithm \ref{alg:tangent-subspace-descent}, we have that whenever
\[ 
\max_{k \in [m]} d(y^{t,0},y^{t,k}) \leq r, 
\]
there exists an orthogonal decomposition $\{P_{t,k}'\}_{k \in [m]}$ of $T_{y^{t,0}} M$ such that for all $k \in [m]$, $P_{t,k}$ and $P_{t,k}'$ have the same rank and $\|P_{t,k}' - P_{t,k}\|_{y^{t,0}} \leq \gamma$ (the induced operator norm on $T_{y^{t,0}} M$). We say such a projection selection rule $\cP$ is \emph{$(\gamma,r)$-gap ensuring}.
\end{assumption}

\begin{remark}\label{rem:gap-ensuring-distance}
As we will prove in Proposition \ref{prop:norm-equivalence}, we can bound the norm-equivalence term if $\max_{k \in [m]} d(y^{t,0},y^{t,k}) \leq r$. One may ask, however, whether the distance parameter $r$ is necessary for our analysis. As it turns out, Section \ref{sec:orthogonal-distance-example} shows that on the manifold of orthogonal matrices $O_n$, there exists a $\cP$ such that Assumption \ref{ass:gap} is satisfied, but the norm-equivalence term is unbounded in general. Thus, the parameter $r$ is critical for bounding the norm equivalence term.

Note that in Euclidean space, $r$ is not required. Furthermore, it is known that small neighbourhoods around any point in $M$ behave like Euclidean space.  Thus, an intuitive reason why the norm-equivalence term is well-behaved when $\max_{k \in [m]} d(y^{t,0},y^{t,k})\leq r$ is because the neighbourhood around $y^{t,0}$ becomes more like usual Euclidean space.

In terms of the analysis, when the distance condition is violated, in Lemma \ref{lemma:deterministic-small-large-step} we will show that another sufficient decrease inequality holds, similar to \eqref{eqn:sufficient-decrease}. Thus, we still make progress for each step. This is the main departure from the traditional analysis of BCD in Euclidean space.
\epr
\end{remark}

\begin{remark}\label{rem:principle-angles}
\citet[p.~331]{Golub13} explain that the quantity $\|P_{t,k}' - P_{t,k}\|_{y^{t,0}}$ is equal to $\sqrt{1-\cos^2(\theta_{\min})}$ where $\theta_{\min}$ denotes the smallest principal angle between the subspaces $\Image(P_{t,k}' )$ and $\Image(P_{t,k})$. The smallest principal angle between tangent subspaces has also been used by \citet{Lewis08} to describe the convergence rates of alternating projection methods on manifolds.
\epr
\end{remark}

We now show how Assumption \ref{ass:gap} leads to bounded norm-equivalence terms.
\begin{proposition}\label{prop:norm-equivalence}
If $\cP$ is $(\gamma,r)$-gap ensuring then at any iteration $t \geq 1$ of Algorithm \ref{alg:tangent-subspace-descent}, it holds that
\[
\left( 1 - \sqrt{m} \gamma \right) \cdot\|\cdot\|_{y^{t,0}} \leq\|\cdot\|_{y^{t,0},\CP} \leq \left( 1 + \sqrt{m} \gamma \right) \cdot \|\cdot\|_{y^{t,0}}.
\]
\end{proposition}

\begin{proof}[Proof of Proposition \ref{prop:norm-equivalence}]
We suppress the dependence on $t$ in the notation throughout this proof. Observe that
\begin{align*}
\|v\|_{y^0}^2 = \sum_{k \in [m]} \|P_k'v\|_{y^0}^2 &\leq \sum_{k \in [m]} \left(\|P_k v\|_{y^0}+\|(P_k'-P_k)v\|_{y^0}\right)^2\\
&\leq\sum_{k \in [m]} \left(\|P_k v\|_{y^0}+\gamma\|v\|_{y^0}\right)^2\\
&=\sum_{k \in [m]} \left[\|P_k v\|_{y^0}^2+2\gamma\|v\|_{y^0}\|P_k v\|_{y^0}+\gamma^2\|v\|_{y^0}^2\right]\\
&=\|v\|_{y^0,\cP}^2+\sum_{k \in [m]} \left[2\gamma\|v\|_{y^0} \|P_k v\|_{y^0} + \gamma^2 \|v\|_{y^0}^2\right]\\
&\leq \|v\|_{y^0,\cP}^2+ 2\sqrt{m}\gamma\|v\|_{y^0} \|v\|_{y^0,\cP} + m \gamma^2 \|v\|_{y^0}^2\\
&= \left(\|v\|_{y^0,\cP} + \sqrt{m} \gamma \|v\|_{y^0} \right)^2\\
 \implies (1-\sqrt{m} \gamma) \|v\|_{y^0} &\leq \|v\|_{y^0,\cP}.
\end{align*}
A similar computation applied to the inequality
\[
\|v\|_{y^0,\cP}^2=\sum_{k \in [m]} \|P_k v\|_{y^0}^2 \leq \sum_{k \in [m]} \left(\|P_k' v\|_{y^0} + \|(P_k-P_k')v\|_{y^0} \right)^2
\]
shows that
\[
\|v\|_{y^0,\cP} \leq (1+\sqrt{m} \gamma) \|v\|_{y^0}
\]
which completes our chain of inequalities.\Halmos
\end{proof}

We are now in a position to complete our analysis. We first verify that \eqref{eq:deterministic-decrease2} holds for any projection selection rule $\CP$ with $C = L_f^2(m-1)/L_{\min}^2$.

\begin{lemma}\label{lemma:deterministic-chaining-bound}
	Suppose $f$ is $L_f$-smooth and Assumption \ref{ass:block-smooth} holds. Then for any $t \geq 1$, the iterates of Algorithm \ref{alg:tangent-subspace-descent}
\[
	\left\|P_{t,k} \grad f(y^{t,0}) - P_{t,k} \Gamma^{y^{t,0}}_{y^{t,k-1}} \grad f(y^{t,k-1}) \right\|_{y^{t,0}}^2 \leq \frac{L_f^2 (m-1)}{L_{\min}^2} \sum_{j \in [k-1]} \|P_{t,j}^{y^{t,j-1}} \grad f(y^{t,j-1}) \|_{y^{t,j-1}}^2.
\]
\end{lemma}

\begin{proof}[Proof of Lemma \ref{lemma:deterministic-chaining-bound}]
	We suppress the dependence of $t$ in the notation throughout the proof. Since $P_k$ is an orthogonal projection,
	\begin{align*}
	\left\|\grad f(y^0)-\Gamma^{y^0}_{y^{k-1}} \grad f(y^{k-1})\right\|_{y^0}^2 &\geq \left\|P_k \grad f(y) - P_k \Gamma^{y^0}_{y^{k-1}} \grad f(y^{k-1}) \right\|_{y^0}^2
	\end{align*}
	Thus, it suffices to examine $\left\|\grad f(y^0)-\Gamma^{y^0}_{y^{k-1}} \grad f(y^{k-1})\right\|_{y^0}^2$ to get the bound. For $k > 1$,
	\[
	\left\|\grad f(y^0)-\Gamma^{y^0}_{y^{k-1}} \grad f(y^{k-1})\right\|_{y^0}^2 \leq L_f d(y^{k-1},y^0)^2
	\]
	by $L_f$-smoothness \eqref{eqn:smooth}. Then
	\[
	d(y^{k-1},y^0)^2\leq  (k-1) \sum_{j \in [k-1]} d(y^{j-1},y^j)^2
	\]
	by the bound $\|\cdot\|_2\geq \sqrt{k-1}\|\cdot\|_1$ where we regard both as norms on $\R^{k-1}$. Thus
	\begin{align*}
	\left\|\grad f(y^0)-\Gamma^{y^0}_{y^{k-1}} \grad f(y^{k-1})\right\|_{y^0}^2& \leq L_f d(y^{k-1},y^0)^2\\
	&\leq L_f^2 (k-1) \sum_{j \in [k-1]} d(y^{j-1},y^j)^2\\
	&\leq L_f^2 (k-1) \sum_{j \in [k-1]} \frac{1}{L_j^2} \|P^{y^{j-1}}_j \grad f(y^{j-1}) \|_{y^{j-1}}^2\\
	&\leq \frac{L_f^2 (k-1)}{L_{\min}^2} \sum_{j \in [k-1]} \|P^{y^{j-1}}_j \grad f(y^{j-1}) \|_{y^{j-1}}^2\\
	\end{align*}
	In the third inequality, we used the fact that $y^j = \Exp_{y^{j-1}}\left( - \frac{1}{L_j} P^{y^{j-1}}_j \grad f(y^{j-1}) \right)$.
\Halmos
\end{proof}

We now present a refined version of \eqref{eqn:sufficient-decrease} in light of Assumption \ref{ass:gap}.
\begin{lemma}\label{lemma:deterministic-small-large-step}
	Suppose that Assumption \ref{ass:block-smooth} holds and that $\cP$ is $(\gamma,r)$-gap ensuring (Assumption \ref{ass:gap}). Let
	\[ \eta := \frac{L_{\min}^2 (1-\sqrt{m} \gamma)^2}{4 L_{\max}\left( L_{\min}^2 + L_f^2(m-1)m \right)}, \quad \eta' := \frac{L_{\min} r^2}{2m}.\]
	Then for any iteration $t \geq 1$ of Algorithm \ref{alg:tangent-subspace-descent},
	\begin{enumerate}
		\item If $\sum_{k \in [m]} d(y^{t,k-1},y^{t,k}) \leq r$, then
		\begin{equation}\label{eq:lemma-small-step}
		f(y^{t,0}) - f(y^{t,m}) \geq \eta \|\grad f(y^{t,0})\|_{y^{t,0}}^2.
		\end{equation}
		\item If $\sum_{k \in [m]} d(y^{t,k-1},y^{t,k}) > r$, then
		\begin{equation}\label{eq:lemma-large-step}
		f(y^{t,0}) - f(y^{t,m}) \geq \eta'.
		\end{equation}
	\end{enumerate}
\end{lemma}

\begin{proof}[Proof of Lemma \ref{lemma:deterministic-small-large-step}]
We fix some $t \geq 1$ and suppress the dependence on $t$ in the notation throughout the proof. 
\begin{enumerate}
	\item If $\sum_{k \in [m]} d(y^{k-1},y^k) \leq r$, then $\max_{k \in [m]} d(y^k,y^0) \leq r$ also by the triangle inequality. The result now follows from  \eqref{eq:deterministic-decrease} in Lemma \ref{lemma:sufficient-decrease-deterministic}, Proposition \ref{prop:norm-equivalence} and Lemma \ref{lemma:deterministic-chaining-bound}.
	
	\item If $\sum_{k \in [m]} d(y^{k-1},y^k) > r$, then
	\begin{align*}
	r^2 < \left( \sum_{k \in [m]} d(y^{k-1},y^k) \right)^2 \leq m \sum_{k \in [m]} d(y^{k-1},y^k)^2&\leq m \sum_{k \in [m]} \frac{1}{L_k^2} \left\|P^{y^{k-1}}_k \grad f(y^{k-1})\right\|_{y^{k-1}}^2\\
	&\leq \frac{m}{L_{\min}} \sum_{k \in [m]} \frac{1}{L_k} \left\|P^{y^{k-1}}_k \grad f(y^{k-1})\right\|_{y^{k-1}}^2 \\
	&\leq \frac{2m}{L_{\min}} (f(y^0) - f(y^m)),
	\end{align*}
	where the second inequality follows since $\|\cdot\|_1 \leq \sqrt{m} \|\cdot\|_2$ for $\ell_1$- and $\ell_2$-norms in $\bbR^m$, the third inequality follows from the definition of $y^k$ in Algorithm \ref{alg:tangent-subspace-descent}, and the final inequality follows from \eqref{eq:deterministic-decrease1} in Lemma \ref{lemma:sufficient-decrease-deterministic}.\Halmos
\end{enumerate}
\end{proof}

\begin{theorem}[Convergence with Deterministic $\cP$]\label{thm:deterministic-descent-rates}
	Let $f:M \to \bbR$ be a $L_f$-smooth function such that Assumption \ref{ass:block-smooth} holds. and that $\cP$ is $(\gamma,r)$-gap ensuring (Assumption \ref{ass:gap}). Suppose that the sequence $\{x^t\}_{t \geq 1}$ is generated from Algorithm \ref{alg:tangent-subspace-descent} with a projection selection rule $\cP$ satisfying Assumption \ref{ass:gap}. Let
	\[ \eta := \frac{L_{\min}^2 (1-\sqrt{m} \gamma)^2}{4 L_{\max}\left( L_{\min}^2 + L_f^2(m-1)m \right)}, \quad \eta' := \frac{L_{\min} r^2}{2m}.\]
	Then the following hold:
	\begin{enumerate}%
		\item For $t > (f(x^1)-f^*)/\eta'$, we have
		\begin{align*}
		&\lim_{t\to\infty} \|\grad f(x^t)\|_{x^t} = 0,\\
		&\min_{s \in [t]} \|\grad f(x^s)\|_{x^s} \leq \max_{k \leq (f(x^1)-f^*)/\eta'} \sqrt{\frac{f(x^1) - f^* - \eta' k}{\eta(t - k)}} = \mathcal{O}\left(m/\sqrt{t}\right). 
		\end{align*}
		\item Any limit point $x^* \in M$ of $\{x^t\}_{t \in \bbN}$ is a stationary point $\grad f(x^*) = 0$.
		\item If in addition $f$ is g-convex and the diameter of $\{x \in M : f(x) \leq f(x^1)\}$ is $R$, then
		\[ 
		f(x^{t+1}) - f^* \leq \frac{f(x^1) - f^*}{1 + \min\left\{ \frac{(f(x^1)-f^*)\eta}{R^2}, \frac{\eta'}{f(x^1)-f^*} \right\} t} = \mathcal{O}\left(m^2/t\right).
		\]
	\end{enumerate}
\end{theorem}
Note that the rates in Theorem \ref{thm:deterministic-descent-rates} are worse than their Euclidean counterparts \cite{BeckTetruashvili2013}. However, we show that in Example \ref{ex:gap-product} that when $M$ is a product manifold (which generalizes the Euclidean setting, see Example \ref{ex:gap-bcd}), Lemma \ref{lemma:deterministic-chaining-bound} can be improved, which results in improved rates in Theorem \ref{thm:deterministic-descent-rates}. These match the ones from the Euclidean setting \cite{BeckTetruashvili2013}.

\begin{proof}[Proof of Theorem \ref{thm:deterministic-descent-rates}]
	Fix some $t \geq 1$. For $s \in [t-1]$, we know from Lemma \ref{lemma:deterministic-small-large-step} that one of the following must hold:
	\[ f(x^s) - f(x^{s+1}) \geq \eta \|\grad f(x^s)\|_{x^s}^2, \quad \text{or} \quad f(x^s) - f(x^{s+1}) \geq \eta'. \]
	Let $R(t) := \left\{ s \in [t-1] : f(x^s) - f(x^{s+1}) \geq \eta' \right\}$. Summing the inequalities over $s \in [t-1]$ gives
	\[ f(x^1) - f(x^t) \geq \eta \sum_{s \in [t-1] \setminus R(t)} \|\grad f(x^s)\|_{x^s}^2 + |R(t)| \eta'. \]
	Note that we must have $|R(t)| \leq (f(x^1)-f(x^t))/\eta' \leq (f(x^1)-f^*)/\eta'$.
	
	\begin{enumerate}%
		\item Fix $\delta > 0$. Since $|R(t)|$ remains bounded, if there exists a subsequence of $\{ \|\grad f(x^t)\|_{x^t} : t \geq 1\}$ which is bounded below by $\delta$, then $\eta \sum_{s \in [t-1] \setminus R(t)} \|\grad f(x^s)\|_{x^s}^2 \to \infty$. Therefore, every subsequence of $\{ \| \grad f(x^t) \|_{x^t} : t \geq 1 \}$ must be less than $\delta$ eventually, so $\lim_{t\to\infty} \|\grad f(x^t)\|_{x^t} = 0$.
		
		The bound above implies
		\[ f(x^1) - f(x^t) - \eta' |R(t)| \geq \eta \sum_{s \in [t] \setminus R(t)} \|\grad f(x^s)\|_{x^s}^2 \geq (t - |R(t)|) \eta \min_{s \in [t]} \|\grad f(x^s)\|_{x^s}^2, \]
		which gives the second result.
		
		\item Since $\|\grad f(x^t)\|_{x^t} \to 0$,  any limit point $x^*$ is a stationary point.
		
		\item For $s \in [t]$, if we have $f(x^s) - f(x^{s+1}) \geq \eta \|\grad f(x^s)\|_{x^s}^2$ then we also have 
		\begin{align*}
		f(x^s) - f(x^*) \leq \la \grad f(x^s), \Exp_{x^s}^{-1}(x^*) \ra_{x^s}\leq \|\grad f(x^s)\|_{x^s} \|\Exp_{x^s}^{-1}(x^*)\|_{x^s}&\leq \|\grad f(x^s)\|_{x^s} d(x^s,x^*)\\
		& \leq \sqrt{(f(x^s) - f(x^{s+1}))/\eta} d(x^s,x^*)\\
		&\leq R \sqrt{(f(x^s) - f(x^{s+1}))/\eta}.
		\end{align*}
		Otherwise, we get $f(x^s) - f(x^{s+1}) \geq \eta'$. Denoting $A_s := f(x^s) - f(x^*)$, these two inequalities can be written as
		\[ \frac{\eta}{R^2} A_s^2 \leq A_s - A_{s+1} \quad \text{or} \quad A_s - A_{s+1} \geq \eta'. \]
		Dividing by $A_s A_{s+1}$ and recognizing that $1 \leq A_s/A_{s+1}$, $A_s A_{s+1} \leq A_1^2$, we have
		\[ \frac{\eta}{R^2} \leq \frac{1}{A_{s+1}} - \frac{1}{A_s} \quad \text{or} \quad \frac{\eta'}{A_1^2} \leq \frac{1}{A_{s+1}} - \frac{1}{A_s} \implies \frac{1}{A_{s+1}} - \frac{1}{A_s} \geq \min\left\{ \frac{\eta}{R^2}, \frac{\eta'}{A_1^2} \right\}. \]
		Summing this over $s \in [t]$, we get
		\[ \frac{1}{A_{t+1}} - \frac{1}{A_1} \geq t \min\left\{ \frac{\eta}{R^2}, \frac{\eta'}{A_1^2} \right\} \implies A_{t+1} \leq \frac{A_1}{1 + \min\left\{ \frac{A_1 \eta}{R^2}, \frac{\eta'}{A_1} \right\} t}, \]
		which gives the result.\Halmos
		
	\end{enumerate}
\end{proof}

\subsection{Examples}\label{sec:deterministic-examples}

We now exhibit a series of selection rules satisfying the gap ensuring condition. The first rule applies to any manifold.
\begin{example}[Parallel Transport Rule]\label{ex:parallel-transport-rule}
For any $t \geq 1$, given $y^{t,0}$, fix an orthogonal decomposition $\{P_{t,k}'\}_{k \in [m]}$ of $T_{y^{t,0}} M$. We will define $\cP$ recursively via the rule
\[
\cP(y^{t,k-1},k)=\Gamma^{y^{t,k-1}}_{y^{t,0}} P_{t,k}' \Gamma_{y^{t,k-1}}^{y^{t,0}}.
\]
It follows that $P_{t,k} = P_{t,k}'$ so $\cP$ is a $(0,\infty)$-gap ensuring rule. Note that since $r = \infty$, the convergence rate bounds are obtained by taking the bounds in Theorem \ref{thm:deterministic-descent-rates} and setting $\eta' = 0$.
\epr
\end{example}

While $\cP$ applies generally to any manifold, the main difficulty in its practical implementation is that on many manifolds it is unclear how, or incredibly expensive, to compute parallel transport maps. For example, on every Stiefel manifold $\St(p,n)$ with the exception to that of the manifold of orthogonal matrices $O_n = \St(n,n)$, there is no known closed form for the parallel transport operator. We provide a non-trivial, and very useful, example of a $(\gamma,r)$-gap ensuring rule for the orthogonal manifold $O_n$. While the construction is conceptually simple, it requires some background knowledge of the manifold structure of $O_n$, so we defer its presentation to Section \ref{sec:orthogonal}.

For the rest of this section, we examine the class of product manifolds, for which the rule described in Example \ref{ex:parallel-transport-rule} can be efficiently computed.

\begin{example}[Product Manifolds]\label{ex:gap-product}
Let $M := M_1\times \ldots \times M_m$ be a Cartesian product of Riemannian manifolds with the endowed product structure. Write $y \in M$ as $y = (y_1,\ldots,y_m)$. The tangent space is $T_y M = \bigoplus_{k \in [m]} T_{y_k} M_k$. Denote the zero element of $T_{y_k} M$ as $0_{y_k}$. Due to the product structure of $M$, it can be easily shown that the exponential map and parallel transport also decomposes into components $M_k$: for any $x = (x_1,\ldots,x_m) \in M$, $v = (v_1,\ldots,v_m) \in T_x M$, 
\begin{align*}
y &=\Exp_{x}(v)=\left(\Exp_{x_1}(v_1), \ldots, \Exp_{x_m}(v_m) \right)\\
\Gamma^y_x&=\left( \Gamma^{y_1}_{x_1},\ldots,\Gamma^{y_m}_{x_m} \right).
\end{align*}

For $y = (y_1,\ldots,y_m) \in M$, $k \in [m]$, we define
\[\cP(y,k) : v = (v_1,\ldots,v_m) \mapsto \left( 0_{y_1},\ldots,0_{y_{k-1}},v_k,0_{y_{k+1}},\ldots,0_{y_m} \right).\]
Observe that $\{\cP(y,k)\}_{k \in [m]}$ is an orthogonal decomposition of $T_y M$ for any $y\in M$. The projection selection rule we use in Algorithm \ref{alg:tangent-subspace-descent} is the same as the one from Example \ref{ex:parallel-transport-rule}: at a point $y^{t,0} \in M$, use the orthogonal decomposition $\{\CP(y^{t,0},k) \}_{k \in [m]}$ and define
\[ P_k^{y^{t,k-1}} := \Gamma_{y^{t,0}}^{y^{t,k-1}} \CP(y^{t,0},k) \Gamma_{y^{t,k-1}}^{y^{t,0}}. \]
As before, this is a $(0,\infty)$-gap ensuring rule. Due to the decomposability of the exponential map and parallel transport operators on the product manifold $M$, it is easily shown that
\[ P_k^{y^{t,k-1}} = \CP(y^{t,k-1},k). \]
Thus, due to the product structure, unlike Example \ref{ex:parallel-transport-rule}, $P_k^{y^{t,k-1}}$ does not require computing parallel transports.

In fact, the product manifold setting offers two more simplifications. Note that if $y = \Exp_{x}(\CP(x,k) v)$, then $y_j = x_j$ for $j \neq k$, and $y_k = \Exp_{x_k}(v_k)$. Thus, at any inner iteration $k \in [m]$ of Algorithm \ref{alg:tangent-subspace-descent}, computing projections of a tangent vector and updating the current iterate is easy; we only need to update the component corresponding to $M_k$.

Furthermore, due to the decomposability of the parallel transport, the constant in Lemma \ref{lemma:deterministic-chaining-bound} can be improved to $L_f^2/L_{\min}^2$ (removing a factor of $m-1$). This means that the $\eta,\eta'$ terms in the bounds from Lemma \ref{lemma:deterministic-small-large-step} and Theorem \ref{thm:deterministic-descent-rates} become $\eta = \frac{L_{\min}^2}{4 L_{\max}\left( L_{\min}^2 + L_f^2 m \right)}$, $\eta' = \infty$. A closer inspection shows that these bounds exactly match those of \citet{BeckTetruashvili2013} for BCD in the Euclidean setting.
\epr
\end{example}

As a special case of this subspace selection rule, we have the selection rules for BCD. 

\begin{example}[Block Coordinate Descent]\label{ex:gap-bcd}
Let $e_i$ denote the $i$-th standard coordinate vector in $\R^n$ and $\{S_k\}_{k \in [m]}$ be a partition of the set $\{e_i\}_{i \in [n]}$. It follows that $\R^n$ is equal to the direct sum of $\bigoplus_{k \in [m]} \Span(S_k)$, which is a product manifold. Furthermore, letting $P_k$ be the projection onto $\Span(S_k)$, the rule from the previous example is $\cP(x,k) = P_k$, which is exactly block coordinate descent. This shows that our analysis can recover the standard analysis for BCD by \citet{BeckTetruashvili2013}.
\end{example}

Another important application of the product setting is Tensor principle component analysis (PCA).

\begin{example}[Tensor PCA]\label{ex:tensor}
Given a tensor $T \in \bbR^{d_1 \times \cdots \times d_m}$, we wish to compute a decomposition of $T$ into tensor products of orthogonal matrices. We do this by solving the following problem:
\begin{align}
\min_{C,U_1,\ldots,U_m} \quad & \left\| T - C \times_1 U_1 \times \cdots \times_m U_m \right\|_F^2\notag\\
\text{s.t.} \quad & C \in \bbR^{n_1 \times \cdots \times n_m}\label{eq:example-tensor}\\
& U_k \in \St(d_k,n_k), \ k \in [m].\notag
\end{align}
In the objective, $\|\cdot\|_F$ is the Frobenius norm on tensors, and $\times_k$ is the mode-$k$ tensor product. The objective is smooth in the variables $(C,U_1,\ldots,U_m)$; we refer to \citet{Xu2015} for detailed definitions and smoothness results.

Of particular interest to us is the domain. Here, $\St(p,n)$ is the \emph{Stiefel manifold} of $n \times p$ matrices (with $p \leq n$) with orthogonal columns:
\[ \St(p,n) = \left\{ U \in \bbR^{n \times p} : U^\top U = I_p \right\}. \]
Thus, the domain is actually a product manifold:
\[ (C,U_1,\ldots,U_m) \in M := \bbR^{n_1 \times \cdots \times n_m} \times \St(d_1,n_1) \times \cdots \times \St(d_m,n_m). \]
To describe the geometry of $M$, it thus suffices to describe the geometry of $\St(p,n)$. We give the key identities here, and refer to \citet{EdelmanEtAl1998} for full details. The tangent space at a matrix $U \in \St(p,n)$ is
\[ T_U \St(p,n) := \left\{ U A + U_{\perp} B : \begin{aligned}
&A = -A^\top \in \bbR^{p \times p},\\
&B \in \bbR^{(n-p) \times p},\\
&U_{\perp} \in \bbR^{n \times (n-p)} \text{ s.t. } U U^\top + U_\perp U_{\perp}^\top = I_n
\end{aligned} \right\}, \]
which is a subset of $\bbR^{n \times p}$. Since $\St(p,n)$ is embedded in $\bbR^{n \times p}$, the gradient of a function $f:\St(p,n) \to \bbR$ at $U$ can be computed as the projection of the Euclidean gradient onto $T_{U} \St(p,n)$ (see \citet[Section 2]{ZhangEtAl2019}). The exponential map is at some $U A + U_{\perp} B \in T_U \St(p,n)$ is
\[ \Exp_U(U A + U_{\perp} B) := \begin{bmatrix}
U & U_{\perp}
\end{bmatrix} \exp\left( \begin{bmatrix}
A & -B^\top \\ B & 0
\end{bmatrix} \right) \begin{bmatrix} I_p \\ 0\end{bmatrix}. \]
Computing this exponential map requires a computing the matrix exponential of a $n \times n$ matrix. If we attempt to run (the generalization of) gradient descent to solve \eqref{eq:example-tensor}, we must compute the exponential map for each $\St(d_k,n_k)$ at each iteration. On the other hand, using tangent subspace descent, we need only compute one of these each iteration.
\epr
\end{example}

\section{Randomized Subspace Selection}\label{sec:randomized}

In this section, we analyze Tangent Subspace Descent when it is equipped with a randomized subspace selection rule $\cP$. Recall that a randomized subspace selection rule $\CP$ can be instantiated in Algorithm \ref{alg:tangent-subspace-descent} by setting $m=1$, and $P_t := \CP(y^{t,0},1)$ is a projection onto a randomly selected subspace of $T_{x^{t-1}} M$. Thus, at iteration $t \geq 1$ in Algorithm \ref{alg:tangent-subspace-descent}, we compute
\begin{equation}\label{eq:randomized-update-prob}
x^t = \Exp_{x^{t-1}} \left( - \eta_t P_t \grad f(x^{t-1}) \right),
\end{equation}
noting that $x^t$ itself is a random variable due to the randomness in $P_t$.

Most existing rules of practical interest draw $P_t$ from a finite collection of projections, such as the ones for randomized BCD \cite{Nesterov2012,FrongilloReid2015} and the orthogonal manifold \cite{ShalitChechik2014}. In Section \ref{sec:randomized-norm}, we propose a new and relatively unrestrictive condition on $\cP$, the $C$-randomized norm condition, that guarantees convergence for smooth functions. This condition is satisfied for the schemes of \cite{Nesterov2012,ShalitChechik2014,FrongilloReid2015} as well as other important classes, demonstrated in Section \ref{sec:randomized-examples}. However, we note that our condition is sufficiently general that it covers the case when $\cP$ draws from an infinite collection of projections. We exploit this in Section \ref{sec:randomized-stiefel}, where we present a new randomized $\cP$ for optimization over the Stiefel manifold $\St(p,n)$ for $p < n$.

\subsection{A Convergent Randomized \texorpdfstring{$\cP$}{Lg} via the \texorpdfstring{$C$}{Lg}-Randomized Norm Condition}\label{sec:randomized-norm}

As before, the aim is to achieve a sufficient decrease inequality similar to \eqref{eqn:sufficient-decrease}, from which the convergence rate follows. In Section \ref{sec:deterministic}, we saw that Assumption \ref{ass:block-smooth} facilitated the selection of stepsizes $\eta_{t,k}$ and achieving sufficient decrease. Since $P_t$ is a random projection drawn from a possibly infinite collection, we will assume a uniform constant $L_f < \infty$ exists such that for \emph{any} $t \geq 1$ and realization $P_t$,
\begin{multline}\label{eqn:smooth-block-randomized}
f\left(\Exp_{x^{t-1}}(P_t v) \right) \leq f(x^{t-1}) + \ip{\grad f(x^{t-1})}{\cP(x^{t-1},k) v}_{x^{t-1}} + \frac{L_f}{2} \left\|P_t v\right\|_{x^{t-1}}^2.
\end{multline}
Since we are considering smooth objective functions $f$, such an upper bound exists from \eqref{eqn:smooth-ub}. From this, we can choose stepsizes $\eta_t$ which give a preliminary sufficient decrease guarantee, similar to Lemma \ref{lemma:sufficient-decrease-deterministic}. (In practice, we can implement $\eta_t$ via backtracking line search, to achieve the condition \eqref{eq:randomized-decrease-block}, which can result in larger steps and quicker convergence.)

\begin{lemma}[Randomized Sufficient Decrease Template]\label{lemma:sufficient-decrease-randomized}
Suppose $f$ is an $L_f$-smooth function so that \eqref{eqn:smooth-block-randomized} holds, and that at each $t \geq 1$ in Algorithm \ref{alg:tangent-subspace-descent}, the stepsize is chosen as $\eta_t := 1/L_f$. Then
\[f(x^{t-1}) - \bbE\left[ f(x^t) \mid x^{t-1} \right] \geq \frac{1}{2 L_f} \bbE\left[ \left\| P_t \grad f(x^{t-1}) \right\|_{x^{t-1}}^2 \mid x^{t-1} \right].\]
\end{lemma}

\begin{proof}[Proof of Lemma \ref{lemma:sufficient-decrease-randomized}]
The stepsize selection $\eta_t = 1/L_f$ minimizes the right hand side of \eqref{eqn:smooth-block-randomized} to get
\begin{equation}\label{eq:randomized-decrease-block}
f(x^{t-1}) - f(x^t) \geq \frac{1}{2 L_f} \left\| P_t \grad f(x^{t-1}) \right\|_{x^{t-1}}^2.
\end{equation}
Taking the expectation over the randomness in $P_t$, we have
\begin{align*}
f(x^{t-1}) - \bbE\left[ f(x^t) \mid x^{t-1} \right] &\geq \frac{1}{2 L_f} \bbE\left[ \left\| P_t \grad f(x^{t-1}) \right\|_{x^{t-1}}^2 \mid x^{t-1} \right].\Halmos
\end{align*}
\end{proof}

In order to get a sufficient decrease guarantee similar to \eqref{eqn:sufficient-decrease}, we must lower bound the right hand side with $\|\grad f(x^{t-1})\|_{x^{t-1}}$. This motivates the $C$-randomized norm condition for $\cP$.
\begin{assumption}[$C$-randomized norm]\label{ass:random-norm}
Suppose there exists $C > 0$ such that at any $t \geq 1$ in Algorithm \ref{alg:tangent-subspace-descent} and for any $v \in T_{x^{t-1}} M$, we have
\[ \sqrt{\bbE\left[ \left\| P_t v \right\|_{x^{t-1}}^2 \mid x^{t-1} \right]} \geq C \|v\|_{x^{t-1}}. \]
Then we say that $\cP$ satisfies the \emph{$C$-randomized norm condition}.
\end{assumption}
Assumption \ref{ass:random-norm} precludes the progressive degeneration of $v\mapsto \bbE[\|P_tv\|_{x^{t-1}}^2 \mid x^{t-1}]$ that undergirds the pathology of Theorem \ref{thm:projection-failure-randomized}. We are now ready to present the convergence analysis.

\begin{theorem}\label{thm:randomized-descent-rates}
	Let $f:M \to \bbR$ be a $L_f$-smooth function satisfying \eqref{eqn:smooth-block-randomized}. Suppose that the sequence $\{x^t\}_{t \geq 1}$ is generated from Algorithm \ref{alg:tangent-subspace-descent} with $\cP$ satisfying Assumption \ref{ass:random-norm} and stepsizes chosen according to Lemma \ref{lemma:sufficient-decrease-randomized}. Then the following hold:
	\begin{enumerate}%
		\item We have
		\begin{align*} 
		&\lim_{t\to\infty} \bbE\left[ \|\grad f(x^t)\|_{x^t} \right] = 0,\\
		&\min_{s \in [t]} \bbE\left[ \|\grad f(x^{s-1})\|_{x^{s-1}} \right] \leq \sqrt{\frac{2L_{f}(f(x^0) - f^*)}{C^2 t}}. 
		\end{align*}
		\item Almost surely, $\|\grad f(x^t)\|_{x^t} \to 0$ and any limit point $x^* \in M$ of $\{x^t\}_{t \in \bbN}$ is a stationary point $\grad f(x^*) = 0$.
		\item If in addition $f$ is g-convex and the diameter of $\{x \in M : f(x) \leq f(x^1)\}$ is $R$, then
		\[ \bbE[f(x^t)] - f^* \leq \frac{2 L_{f} R^2 (f(x^0) - f^*)}{2 L_{f} R^2 + C^2 (f(x^0) - f^*) t}.\]
	\end{enumerate}
\end{theorem}

\begin{proof}[Proof of Theorem \ref{thm:randomized-descent-rates}]
\begin{equation}\label{eq:sufficient-decrease-randomized}
f(x^{t-1}) - \bbE[f(x^t) \mid x^{t-1}] \geq \eta \|\grad f(x^{t-1})\|_{x^{t-1}}^2.
\end{equation}
\begin{enumerate}
	\item Take expectations in \eqref{eq:sufficient-decrease-randomized} to get $\eta \bbE[\|\grad f(x^{s-1})\|_{x^{s-1}}^2] \leq \bbE[f(x^{s-1})] - \bbE[f(x^s)]$. Summing this for $s \in [t]$ gives
	\[ \eta \sum_{s \in [t]} \bbE[\|\grad f(x^{s-1})\|_{x^{s-1}}^2] \leq \bbE[f(x^0)] - \bbE[f(x^t)] \leq f(x^0) - f^*. \]
	This holds for any $t$, and since $\eta > 0$, we must have $\bbE[\|\grad f(x^t)\|_{x^t}^2] \to 0$. Furthermore, observe that
	\[ \eta t \min_{s \in [t]} \bbE[\|\grad f(x^{s-1})\|_{x^{s-1}}^2] \leq \eta \sum_{s \in [t]} \bbE[\|\grad f(x^{s-1})\|_{x^{s-1}}^2] \leq f(x^0) - f^*, \]
	from which the second result follows by Jensen's inequality.
	
	\item Almost sure convergence of $\|\grad f(x^t) \|_{x^t} \to 0$ follows by taking $\CF_t$ to be the filtration generated by $\{x^1,\ldots,x^t\}$, $a_t = f(x_t)$ and $b_t = \eta \|\grad f(x^t)\|_{x^t}^2$. Then $a_t \geq f^*$ and integrability of $a_t$ holds since $a_t = f(x^t) \leq f(x^1)$, which holds since by construction $f(x^t)$ is monotonically decreasing. Then \eqref{eq:sufficient-decrease-randomized} is rewritten as $b_t \leq a_t - \bbE[a_t \mid \CF_t]$, which also shows integrability of $b_t$. Then apply Proposition \ref{prop:random-variable-convergence} (see Appendix \ref{sec:appendix}) to deduce $\|\grad f(x^t) \|_{x^t} \to 0$ almost surely. Then any limit point is a stationary point since $f$ is continuously differentiable.
	
	\item Since $f$ is g-convex, we have for any minimizer $x^*$, $f(x^{s-1}) - f^* \leq R \|\grad f(x^t)\|_{x^t}$. This and \eqref{eq:sufficient-decrease-randomized} give $\frac{\eta}{R^2} (f(x^{s-1}) - f^*)^2 \leq f(x^{s-1}) - \bbE[f(x^s) \mid x^{s-1}]$. By Jensen's inequality, the expectation of the left hand side is greater than or equal to $\frac{\eta}{R^2} \left( \bbE[f(x^{s-1})] - f^* \right)^2$, hence $\frac{\eta}{R^2}\left( \bbE[f(x^{s-1})] - f^* \right)^2 \leq \bbE[f(x^{s-1})] - \bbE[f(x^s)]$. Write $A_s = \bbE[f(x^s)] - f^*$, hence we have $\frac{\eta}{R^2} A_s^2 \leq A_{s-1} - A_s$. Dividing by $A_{s-1} A_s$ and recognizing that $1 \leq A_{s-1}/A_s$, $A_{s-1} A_s \leq A_0^2$, we have
	\[ \frac{\eta}{R^2} \leq \frac{1}{A_s} - \frac{1}{A_{s-1}}. \]
	Summing this over $s \in [t]$, we get
	\[ \frac{1}{A_t} - \frac{1}{A_0} \geq t \frac{\eta}{R^2} \implies A_t \leq \frac{A_0}{1 + \frac{A_0 \eta}{R^2} t}, \]
	which gives the result.\Halmos
\end{enumerate}
\end{proof}
Note that there is no explicit dependence on $m$ in the bounds of Theorem \ref{thm:randomized-descent-rates}. In fact, this dependence is hidden in the constant $C$. The constant $C$ is affected by how the probabilities $p_k$ are chosen, as well as $m$. As we will see, for natural choices of probabilities, the dependence on $m$ will be similar to deterministic TSD for product manifolds (Example \ref{ex:gap-product}).

\subsection{Examples}\label{sec:randomized-examples}

We now elaborate a few important examples of $C$-randomized norm ensuring subspace selection rules as well as our new subspace selection rule for a general Stiefel manifold.
We begin by providing a general randomized selection rule which works for any manifold $M$.

\begin{example}\label{ex:randomized-general}
At each $x^{t-1} \in M$, take any orthogonal decomposition $\{\cP(x^{t-1},k)\}_{k \in [m]}$ of $T_x M$. Let $P_t = \cP(x^{t-1},k)$ with probability $p_k > 0$, where $\sum_{k \in [m]} p_k =1$. Then we can take $C = \sqrt{\min_{k \in [m]} p_k}$ in Assumption \ref{ass:random-norm}: for any $v\in T_{x^{t-1}} M$, we see
\[
\bbE\left[ \|P_tv\|_{x^{t-1}}^2 \mid x^{t-1} \right] = \sum_{k \in [m]} p_k \|\cP(x^{t-1},k) v\|_{x^{t-1}}^2 \geq \left(\min_{k \in [m]} p_k\right) \sum_{k \in [m]} \|\cP(x^{t-1},k) v\|_{x^{t-1}}^2= \left(\min_{k \in [m]} p_k\right) \|v\|_{x^{t-1}}^2.
\]
To optimize the bounds in Theorem \ref{thm:randomized-descent-rates}, we make $C$ as large as possible, which is done by taking $p_k = 1/m$, so $C = \sqrt{1/m}$. The bounds then become $\cO(\sqrt{m/t})$ and $\cO(m/t)$ respectively. These bounds match the rates from Example \ref{ex:gap-product} on product manifolds (but with slightly better hidden constants).
\epr
\end{example}

\begin{example}[Product Manifolds and Randomized Coordinate Descent]
Consider again $M=M_1\times\ldots\times M_m$, the product manifold setting from Example \ref{ex:gap-product}. Since for any $x = (x_1,\ldots,x_m)$, $T_x M = \bigoplus_{k \in [m]} T_{x_k} M$, so this yields a natural orthogonal decomposition
\[ \cP(x^{t-1},k) : v \mapsto \left( 0_{x_1^{t-1}}, \ldots, 0_{x_{k-1}^{t-1}}, v_k, 0_{x_{k+1}^{t-1}}, \ldots, 0_{x_m^{t-1}} \right), \quad k \in [m]. \]

Randomized BCD is instantiated in the following way: set $M=\R^n$, $\{S_k\}_{k \in [m]}$ is a partition of $\{e_i\}_{i \in [n]}$, and $M_k=\Span(S_k)$ for $k \in [m]$.
\epr
\end{example}

\begin{example}[Orthogonal Manifold]\label{ex:orthogonal-randomized}
We consider $O_n=St(n,n)$, which has tangent space that can be nicely expressed as
\[
T_X O_n=\Span\{XH_{ij}:1\leq i,j\leq n\}, \quad H_{ij}=e_i e_j^\top - e_j e_i^\top,
\]
for any $X \in O_n$. The collection of tangent vectors $\{X H_{ij}: 1\leq i,j\leq n\}$ is clearly orthogonal under the Riemannian metric $\ip{X H_{ij}}{X H_{i'j'}}_X = \frac{1}{2}\Tr\left( (X H_{ij})^\top X H_{i'j'} \right)$.  We let $P_t = \CP(X^{t-1},ij)$ be the projection onto $\Span(\{X^{t-1} H_{ij}\})$ for each $1 \leq i < j \leq n$, which is chosen with probability $p_{ij} > 0$, where $\sum_{1 \leq i < j \leq n} p_{ij} = 1$. Then $\CP$ satisfies Assumption \ref{ass:random-norm} with $C = \sqrt{\min_{1 \leq i < j \leq n} p_{ij}}$. This is the subspace selection rule of \citep{ShalitChechik2014}.
\epr
\end{example}

\subsubsection{Randomized rule for the Stiefel manifold}\label{sec:randomized-stiefel}

We can significantly generalize the latter example to $\St(p,n)$ with $p<n$. Given a matrix $U \in \bbR^{n \times p}$, denote $\Ker(U^\top) \subseteq \bbR^n$ to be the kernel of $U^\top$. The tangent space $T_U \St(p,n)$ from Example \ref{ex:tensor} can be written as
\[ T_U \St(p,n) = \left\{ U A + \begin{bmatrix} b_1 & \cdots & b_p \end{bmatrix} : \begin{aligned}
&A = -A^\top \in \bbR^{p \times p},\\
&b_\ell \in \Ker(U^\top), \ \ell \in [p]
\end{aligned} \right\}. \]
Furthermore, given $V,V' \in T_U \St(p,n)$, the Riemannian metric is $\la V,V' \ra_U =\Tr(V^\top (I_n - U U^\top/2) V')$.

We consider the non-orthogonal subspace decomposition of $T_U \St(p,n)$:
\[ D = \left\{ U H_{ij} : 1 \leq i < j \leq p \right\} \cup \left( \bigcup_{\ell \in [p]} \left\{ v e_{\ell,p}^\top : v \in \Ker(U^\top), \|v\|_2 = 1 \right\} \right), \]
where $e_{\ell,p}$ is the $\ell$-th standard basis vector in $\bbR^p$. Note that $\Span(D) = T_U \St(p,n)$. Our main idea is draw a random tangent vector $V$ from $D$, use it to define a projection onto the subspace spanned by $V$, then show the random projection satisfies Assumption \ref{ass:random-norm}.

We define the following distribution over $D$. Suppose we have a distribution over the indices $\{(i,j) : 1 \leq i < j \leq p\} \cup [p]$ where a pair $(i,j)$ is drawn with probability $p_{ij} > 0$ and $\ell$ is drawn with probability $p_{\ell} > 0$, so $\sum_{1\leq i<j\leq p} p_{ij} + \sum_{\ell \in [p]} p_\ell = 1$. We now draw a random element from $\{(i,j) : 1 \leq i < j \leq p\} \cup [p]$. If it is a pair $(i,j)$, then the element drawn from $D$ is $V = U H_{ij}$. If it is a single index $\ell$, then we randomly draw from the set $\left\{ v e_{\ell,p}^\top : v \in \Ker(U^\top) \right\}$ by drawing $v$ uniformly at random over $\{v : \|v\|_2=1, v \in \Ker(U^\top)\}$ and setting $V = v e_{\ell,p}^\top$. In practice, drawing $v$ can be implemented as $z \sim N(0,I_n)$, $v' = (I_n - U U^\top) z$, $v = v'/\|v'\|_2$. Since $U$ has orthonormal columns, $I_n - U U^\top$ is simply the projection matrix onto $\Ker(U^\top)$.

\begin{theorem}\label{thm:stiefel-selection}
Fix any $U\in \St(p,n)$, and let $\cP(U)$ be the random projection onto the subspace $\{\alpha V : \alpha \in \bbR\} \subset T_U \St(p,n)$, where $V$ is drawn randomly according to the distribution described above. Let $C := \sqrt{\min\left\{ \min_{1 \leq i < j \leq p} p_{ij}, \min_{\ell \in [p]} p_\ell/(n-p) \right\}}$. Then for any $W \in T_U \St(p,n)$,
\[ \bbE\left[ \left\| \cP(U) W \right\|_U^2 \right] \geq C^2 \|W\|_U^2. \]
\end{theorem}

\begin{proof}[Proof of Theorem \ref{thm:stiefel-selection}]
Denote $W = U A + \begin{bmatrix} b_1 & \cdots & b_p \end{bmatrix}$ and observe that
\[ \|W\|_U^2 = \sum_{1\leq i<j\leq p} a_{ij}^2 + \sum_{\ell \in [p]} \|b_\ell\|_2^2. \]
Given $V \in T_U \St(p,n)$, the projection of $W$ onto $\{\alpha V : \alpha \in \bbR\}$ is given by
\[ \frac{\la V, W \ra_U}{\|V\|_U^2} V. \]
When $V = U H_{ij}$, $\|V\|_U^2 = 1$, so the projection is
\[
\Tr\left( H_{ij}^\top U^\top (I - U U^\top/2) W \right) U H_{ij} = \frac{1}{2} \Tr\left( H_{ij}^\top U^\top W \right) U H_{ij} = \frac{1}{2} \Tr\left( H_{ij}^\top A \right) U H_{ij} = a_{ij} U H_{ij}.
\]
When $V = v e_{\ell,p}^\top$ and $v \in \Ker(U^\top)$, $\|v\|_2=1$, $\|V\|_U^2 = 1$, so the projection is
\[ \Tr\left( e_{\ell,p} v^\top (I - U U^\top/2) W \right) v e_{\ell,p}^\top = \Tr\left( e_{\ell,p} v^\top W \right) v e_{\ell,p}^\top = (v^\top b_\ell) v e_{\ell,p}. \]

We then have
\[ \bbE\left[ \left\| \cP(U) W \right\|_U^2 \right] = \sum_{1 \leq i < j \leq p} p_{ij} a_{ij}^2 + \sum_{\ell \in [p]} p_\ell \bbE\left[ \left( v^\top b_\ell \right)^2 \right], \]
where the expectations are over the randomness in $v$. To evaluate the expectation, for convenience write $b := b_\ell$, and letting $\bar{u}_1,\ldots,\bar{u}_{n-p}$ be an orthonormal basis for $\Ker(U^\top)$, write $b = \sum_{k \in [n-p]} \beta_k \bar{u}_k$. Clearly $\|b\|_2^2 = \sum_{k \in [n-p]} \beta_k^2$. Then
\[
\bbE\left[ \left( v^\top b \right)^2 \right] = \bbE\left[ \left( \sum_{k \in [n-p]} \beta_k v^\top \bar{u}_k \right)^2 \right]= \sum_{k \in [n-p]} \beta_k^2 \bbE\left[ \left( v^\top \bar{u}_k \right)^2 \right] + 2 \sum_{1 \leq k < k' \leq n-p} \beta_k \beta_{k'} \bbE\left[ \left( v^\top \bar{u}_k \right) \left( v^\top \bar{u}_{k'} \right) \right].
\]
Since $\bar{u}_1,\ldots,\bar{u}_{n-p}$ are an orthonormal basis for $\Ker(U^\top)$ and $v \in \Ker(U^\top)$, we have
\[ 1 = \bbE[\|v\|_2^2] = \sum_{k \in [n-p]} \bbE\left[ \left( v^\top \bar{u}_k \right)^2 \right], \]
and since $v$ is spherically symmetric in $\Ker(U^\top)$, $v^\top \bar{u}_k$ are identically distributed for $k \in [n-p]$. Thus $\sum_{k \in [n-p]} \beta_k^2 \bbE\left[ \left( v^\top \bar{u}_k \right)^2 \right] = \frac{1}{n-p} \|b\|_2^2$. Furthermore, due to symmetry, the pairs $(v^\top \bar{u}_k, v^\top \bar{u}_{k'})$ and $(-v^\top \bar{u}_k, v^\top \bar{u}_{k'})$ have the same distribution, thus
\[ \bbE\left[ \left( v^\top \bar{u}_k \right) \left( v^\top \bar{u}_{k'} \right) \right] = \bbE\left[ \left( -v^\top \bar{u}_k \right) \left( v^\top \bar{u}_{k'} \right) \right] = 0. \]
Therefore,
\[
\bbE\left[ \left( v^\top b \right)^2 \right] = \frac{1}{n-p} \|b\|_2^2\implies \bbE\left[ \left\| \cP(U) W \right\|_U^2 \right] = \sum_{1 \leq i < j \leq p} p_{ij} a_{ij}^2 + \frac{1}{n-p} \sum_{\ell \in [p]} p_\ell \|b\|_2^2 \geq C^2 \|W\|_U^2.\Halmos
\]
\end{proof}

In Algorithm \ref{alg:tangent-subspace-descent}, at each $t \geq 1$ we can choose the random projection $P_t$ according to the distribution $\CP(U)$ outlined above with $U = U^{t-1}$, which will satisfy Assumption \ref{ass:random-norm} by Theorem \ref{thm:stiefel-selection}

\section{Tangent Subspace Descent on the Orthogonal Manifold}\label{sec:orthogonal}

In this section, we consider optimization on the manifold of orthogonal $n \times n$ matrices
\[ 
M := O_n = \left\{ Y \in \bbR^{n \times n} : Y^\top Y = Y Y^\top  = I_n \right\}. 
\]
We first briefly describe the manifold structure (see \citet{EdelmanEtAl1998} for full details). At a given $Y \in O_n$, the tangent space is
\[ T_Y O_n := \left\{ Y A : A \in \Skew_n \right\}, \quad \Skew_n := \left\{ A \in \bbR^{n \times n} : A = -A^\top \right\}. \]
Since $M$ is an embedded submanifold of $\bbR^{n \times n}$, it inherits the usual Euclidean metric
\[ 
g_Y(YA,YB) := \Tr((YA)^\top (YB)) = \Tr(A^\top B). 
\]
Geodesics originating from $Y \in M$ in the direction $Y C \in T_Y M$ are defined by
\[ 
t \mapsto \gamma(t) := Y \Expm(t C), 
\]
where $\Expm$ is the matrix exponential operator. Consequently, the exponential map is
\[ 
\Exp_Y(Y C) := Y \Expm(C). 
\]
Parallel transport of a tangent vector $YA \in T_Y M$ along the geodesic $t \mapsto Y \Expm(t C)$ is given by
\[ 
\Gamma_{Y,C}(t;A) := Y \Expm(t C) \Expm(t C/2)^\top A \Expm(t C/2) \in T_{Y \Expm(t C)} O_n. 
\]
Given a smooth function $\bar{f}:\bbR^{n \times n} \to \bbR$, let $f:O_n \to \bbR$ be its restriction to $O_n$, and $X \in O_n \subset \bbR^{n \times n}$. \citet[Eqs. (3.35), (3.37)]{AbsilEtAl2007book} show us how to compute $\grad f(X)$:
\begin{equation}\label{eq:orthogonal-gradient}
\grad f(X) = \frac{1}{2} X \left( X^\top \grad \bar{f}(X) - \grad \bar{f}(X)^\top X \right),
\end{equation}
where $\grad \bar{f}(X)$ is the usual Euclidean gradient of $\bar{f}$ at $X$. Note that $\frac{1}{2} \left( X^\top \grad \bar{f}(X) - \grad \bar{f}(X)^\top X \right) \in \Skew_n$.

\subsection{Deterministic Tangent Subspace Descent on $O_n$}

The goal of this section is to provide a deterministic subspace selection rule $\CP$ that satisfies Assumption \ref{ass:gap}, and therefore tangent subspace descent with this rule converges according to Theorem \ref{thm:deterministic-descent-rates}. Our subspace decompositions will be based on partitions of the following basis for $\Skew_n$:
\[ \CB := \left\{ \frac{1}{\sqrt{2}} (e_i e_j^\top - e_j e_i^\top) : 1 \leq i < j \leq n \right\}. \]
Let $\{\CB_k\}_{k \in [m]}$ be some partition of $\CB$. Define
\[ \Skew_{n,k} := \Span(\CB_k), \quad B_k := \Proj_{\Skew_{n,k}} : \Skew_n \to \Skew_{n,k}, \]
where $\Proj$ is the usual Euclidean projection. Then, given $Y \in O_n$ and $k\in[m]$, define projections $P_k^Y : T_Y O_n \to T_Y O_n$ as the projection onto the subspace,
\[
S_k^Y := \left\{ Y A : A \in \Skew_{n,k} \right\} \subset T_Y O_n,
\]
which is computed as
\[ YA \mapsto P_k^Y(Y A) := Y B_k(A). \]
We propose the following subspace selection rule $\CP$ for the orthogonal manifold $O_n$: given $Y^0,\ldots,Y^{k-1}$ obtained at step $k$ of the inner loop of Algorithm \ref{alg:tangent-subspace-descent}, define
\begin{equation}\label{eq:orthogonal-subspace-selection}
\CP\left( \left\{ Y^i \right\}_{i=0}^{k-1}, k \right) := P_k^{Y^{k-1}}.%
\end{equation}
Our main result in this section is to show that $\CP$ satisfies Assumption \ref{ass:gap}.

\begin{theorem}\label{thm:orthogonal-gap-ensuring}
	Fix any $\beta \in (0,1)$. The projection selection rule $\CP:=\{P_k\}_{k \in [m]}$ defined in \eqref{eq:orthogonal-subspace-selection}
	is $\left( \max_{k\in[m]}|\CB_k| \sqrt{1-\beta^2}, 2 \log(1 + \sqrt{1-\beta}) \right)$-gap ensuring. Thus, when $\beta > \sqrt{1 - \frac{1}{m \max_{k\in[m]} |\CB_k|}}$, we have
	\[ \|Y^0 A\|_{Y^0} \leq \frac{1}{1 - \sqrt{m(1-\beta^2)} \max_{k\in[m]}|\CB_k|} \left\| Y^0 A \right\|_{Y^0,\CP}, \]
provided $d(Y^0,Y^k) \leq 2 \log(1 + \sqrt{1-\beta})$ for each $k \in [m]$. 
\end{theorem}

\begin{proof}[Proof of Theorem \ref{thm:orthogonal-gap-ensuring}]
The second part follows from the first part and Proposition \ref{prop:norm-equivalence}, thus it suffices to prove the first part.

To do this, we assume that $d(Y^0,Y^i) \leq 2 \log(1 + \sqrt{1-\beta})$ for all $i \in [k-1]$, and we provide an orthogonal subspace decomposition $\left\{ P_k' \right\}_{k \in [m]}$ of $T_{Y^0} O_n$ such that $\|P_k - P_k'\| \leq \max_{k\in[m]}|\CB_k| \sqrt{1-\beta^2}$, where $\|P_k - P_k'\|$ is the operator norm. A natural choice is to set
\[ 
P_k' := P_k^{Y^0}, \quad k \in [m]. 
\]
Let $C_k \in \Skew_n$ be matrices such that $\Exp_{Y^{k-1}}(Y^{k-1} C_k) = Y^0$, and $\Gamma_{Y^{k-1}}^{Y^0}$ be the parallel transport along the geodesic $Y^{k-1} \Expm(t C_k)$. According to Lemma \ref{lemma:projection}, $P_k$ is the projection onto the subspace
\[
S_k := \left\{ \Gamma_{Y^{k-1}}^{Y^0} (Y^{k-1} A) : A \in \Skew_{n,k} \right\} = \left\{ Y^0 \Expm(C_k/2)^\top A \Expm(C_k/2) : A \in \Skew_{n,k} \right\}.
\]
We compare this with $P_k'$, which is the projection onto the subspace $S_k' :=\left\{ Y^0 A : A \in \Skew_{n,k} \right\}$. The result now follows from Proposition \ref{prop:projection-norm}, provided we can show that for any $A \in \CB_k$,
\[ \Tr\left( A^\top \Expm(C_k/2)^\top A \Expm(C_k/2) \right) \geq \beta. \]
For convenience, denote $C = C_k$. We have
\begin{align*}
\Tr\left( A^\top \Expm(C/2)^T A \Expm(C/2) \right) &= \Tr\left( (\Expm(C/2) A)^\top A \Expm(C/2) \right)\\
&= \Tr\left( \left(\left(I +  \sum_{i=1}^\infty \frac{C^i}{2^i i!} \right) A \right)^\top A \left(I +  \sum_{i=1}^\infty \frac{C^i}{2^i i!} \right) \right)\\
& = \Tr(A^\top A) + \sum_{i,j=1}^\infty \frac{1}{2^{i+j} i! j!} \Tr\left( (C^i A)^\top A C^j \right).
\end{align*}
Notice that using Cauchy-Schwarz and the fact that $\|X X'\|_F \leq \|X\|_F \|X'\|_F$, we have
\begin{align*} 
- \|A\|_F^2 \|C\|_F^{i+j} \leq - \left\| C^i A \right\|_F \left\| A C^j \right\|_F \leq \Tr\left( (C^i A)^\top A C^j \right)\leq \left\| C^i A \right\|_F \left\| A C^j \right\|_F \leq \|A\|_F^2 \|C\|_F^{i+j},
\end{align*}
which implies
\begin{align*}
- \|A\|_F^2 \left( \exp(\|C\|_F/2)-1 \right)^2 = -\|A\|_F^2 \sum_{i,j=1}^\infty \frac{1}{2^{i+j} i! j!} \|C\|_F^{i+j}&\leq \sum_{i,j=1}^\infty \frac{1}{2^{i+j} i! j!} \Tr\left( (C^i A)^\top A C^j \right)\\
&\leq \|A\|_F^2 \sum_{i,j=1}^\infty \frac{1}{2^{i+j} i! j!} \|C\|_F^{i+j}\\
& = \|A\|_F^2 \left( \exp(\|C\|_F/2)-1 \right)^2.
\end{align*}
Therefore
\begin{align*}
\Tr\left( A^\top \Expm(C/2)^T A \Expm(C/2) \right) &= \Tr(A^\top A) + \sum_{i,j=1}^\infty \frac{1}{2^{i+j} i! j!} \Tr\left( (C^i A)^\top A C^j \right)\\
&= \|A\|_F^2 + \sum_{i,j=1}^\infty \frac{1}{2^{i+j} i! j!} \Tr\left( (C^i A)^\top A C^j \right)\\
&\geq \|A\|_F^2 - \|A\|_F^2 \left( \exp(\|C\|_F/2)-1 \right)^2\\
&= \|A\|_F^2 \left( 1- \left( \exp(\|C\|_F/2)-1 \right)^2 \right).
\end{align*}
Note that $\|A\|_F = 1$ by definition, and since $d(Y^0,Y^{k-1}) \leq 2 \log(1 + \sqrt{1-\beta})$, we can choose $C$ such that $\|C\|_F \leq 2 \log(1 + \sqrt{1-\beta})$, or equivalently, $1- \left( \exp(\|C\|_F/2)-1 \right)^2 \geq \beta$, which gives the result.
\Halmos
\end{proof}

Convergence of tangent subspace descent with $\CP$ from \eqref{eq:orthogonal-subspace-selection} now follows from Theorem \ref{thm:deterministic-descent-rates}.

\subsection{Givens Rotations and Computational Complexity}\label{sec:orthogonal-givens}

We now show that, when the partition $\{\CB_k\}_{k \in [m]}$ is chosen appropriately, tangent subspace descent updates take a very simple form. Given $Y^0$, an inner loop of Algorithm \ref{alg:tangent-subspace-descent} computes the sequence $Y^1,\ldots,Y^m$ recursively as
\[ Y^k := \Exp_{Y^{k-1}} \left( - \eta_{t,k} P_k^{Y^{k-1}} \grad f(Y^{k-1}) \right), \quad k \in [m]. \]
Let $A^k = -\eta_{t,k} B_k(\grad f(Y^{k-1}))$. Then
\[ Y^k = Y^{k-1} \Expm(A^k). \]
We now describe the complexity of computing
\[\Expm(A^k) = I_n + \sum_{p=1}^\infty \frac{1}{p!} (A^k)^p.\]
Define $H_{ij} :=  e_i e_j^\top - e_j e_i^\top$. Write each index set $(i,j)$ defining $\CB_k$ as $I_k$, i.e., $\CB_k = \{H_{ij}/\sqrt{2} : (i,j) \in I_k\}$. Now write
\[ A^k = \sum_{(i,j) \in I_k} a_{ij}^k H_{ij}. \]
In order to efficiently compute $\Expm(A^k)$, we make the following assumption on $B_k$.

\begin{assumption}\label{ass:orthogonal-blocks}
For each $k \in [m]$, if pairs $(i,j),(i',j') \in I_k$ are distinct, then $i,i',j,j'$ are also distinct. Consequently, $H_{ij} H_{i'j'}$ is the zero matrix.
\end{assumption}
Assumption \ref{ass:orthogonal-blocks} ensures that $(A^k)^p = \sum_{(i,j) \in I_k} (a_{ij}^k)^p H_{ij}^p$, and $H_{ij}^p$ takes an easily computable form, which results in the following expression for each $(i,j)$-submatrix of $\Expm(A^k)$ where $(i,j) \in I_k$:
\begin{equation}\label{eq:orthongal-givens}
\begin{bmatrix}
\Expm(A^k)_{ii} & \Expm(A^k)_{ij}\\ \Expm(A^k)_{ji} & \Expm(A^k)_{jj}
\end{bmatrix} = \begin{bmatrix}
\cos(a_{ij}^k) & \sin(a_{ij}^k)\\ -\sin(a_{ij}^k) & \cos(a_{ij}^k)
\end{bmatrix}.
\end{equation}
This submatrix is known as a \emph{Givens rotation matrix}. All other off-diagonal entries of $\Expm(A^k)$ are $0$, and all other diagonal entries are $1$.

\begin{proposition}\label{prop:orthogonal-complexity}
Under Assumption \ref{ass:orthogonal-blocks}, the complexity of one inner loop of Algorithm \ref{alg:tangent-subspace-descent} is $O(n^2(n-1)/2)$.
\end{proposition}

\begin{proof}[Proof of Proposition \ref{prop:orthogonal-complexity}]
Under Assumption \ref{ass:orthogonal-blocks}, when $A^k \in \Span(\CB_k)$, computation of $\Expm(A^k)$ takes time $O(|B_k|)$. Updating $Y^k = Y^{k-1} \Expm(A^k)$ takes time $O(n|B_k|)$, since for each pair $(i,j) \in I_k$, we must perform two vector additions with two columns of $Y^{k-1}$. Overall, the complexity of one inner loop of Algorithm \ref{alg:tangent-subspace-descent} is $O\left( n \sum_{k \in [m]} |\CB_k| \right) = O(n^2(n-1)/2)$.\Halmos
\end{proof}

Note that while the $O(n^3)$ complexity in Proposition \ref{prop:orthogonal-complexity} is of the same order as computing a matrix exponential for a general matrix $\Expm(A)$, which is required for performing Riemannian gradient descent in $O_n$, the advantage of tangent subspace descent is the fact that progress can be made every $O(n|\CB_k|)$ flops, whereas Riemannian gradient descent must wait until $O(n^3)$ flops have elapsed to make progress. This is particularly advantageous when $n$ is large and the $\CB_k$ are chosen to be small.

\subsection{Necessity of Distance-Dependent Bounds}\label{sec:orthogonal-distance-example}

We now show by a counter-example in $O_n$ why Assumption \ref{ass:block-smooth}, the $(\gamma,r)$ gap-ensuring condition, and the bound in Proposition \ref{prop:norm-equivalence} in general requires the dependence on a distance parameter $r$, as promised in Remark \ref{rem:gap-ensuring-distance}.

Suppose we are given iterates $Y^0,Y^1,\ldots,Y^m$ generated from an inner loop of Algorithm \ref{alg:tangent-subspace-descent}, and that we can write $Y^0 = \Exp_{Y^{k-1}}(Y^{k-1} C_k) = Y^{k-1} \Expm(C_k)$ for some $C_k \in \Skew_n$. The proof of Theorem \ref{thm:orthogonal-gap-ensuring} (namely, the application of Lemma \ref{lemma:projection}) showed that the (non-orthogonal) subspace decomposition $\{P_k\}_{k \in [m]}$ of $T_{Y^0} O_n$ is given by
\[ P_k = \Proj_{S_k}, \quad S_k := \left\{ Y^0 \Expm(C_k/2)^\top A \Expm(C_k/2) : A \in \Skew_{n,k} \right\}. \]
The norm under consideration is
\[ \left\| Y^0 A \right\|_{Y^0,\CP} := \sqrt{ \sum_{k \in [m]} \left\| P_k Y^0 A \right\|_{Y^0}^2}. \]
The proof of Theorem \ref{thm:orthogonal-gap-ensuring} showed that when each $C_k$ is sufficiently small, we can show that Assumption \ref{ass:gap} is satisfied, and hence bound $\sup_{A \in \Skew_n} \|Y^0 A\|_{Y^0}/\|Y^0 A\|_{Y^0,\CP}$.

Our goal in this section is to show that when there is no bound on the skew-symmetric matrices $C_k$, $\sup_{A \in \Skew_n} \|Y^0 A\|_{Y^0}/\|Y^0 A\|_{Y^0,\CP}$ can be arbitrarily large. For convenience (although this is not necessary), we set each $\CB_k$ to be a singleton member of $\CB$, so $m=n(n-1)/2$. We show that there exists $k > 1$ and some adversarial choice of $C_k$ that can make $S_k = S_1$, hence $\Span\left( \bigcup_{k \in [m]} S_k \right) \subsetneq T_{Y^0} O_n$, hence $\sup_{A \in \Skew_n} \|Y^0 A\|_{Y^0}/\|Y^0 A\|_{Y^0,\CP} = \infty$.

Note that for $k=1$, $Y^{k-1} = Y^0$, so we can take $C_1 = 0$, hence $S_1 = \{ \alpha Y^0 H_{ij} : \alpha \in \bbR \}$ for some pair $(i,j)$. Consider some $k \neq 1$, where
\[
S_k = \left\{ \alpha Y^0 \Expm(C_k/2)^\top H_{i'j'} \Expm(C_k/2) : \alpha \in \bbR \right\}
\] 
for some other pair $(i',j')$ such that $i,j,i',j'$ are all unique. We show that we can choose $C_k$ in such a way that
\[\Expm(C_k/2)^\top H_{i'j'} \Expm(C_k/2) = H_{ij}.\]
This can be done by setting
\[ C_k = \pi (H_{ii'} + H_{jj'}). \]
According to the arguments in Section \ref{sec:orthogonal-givens}, particularly \eqref{eq:orthongal-givens}, $\Expm(C_k/2)$ is the identity matrix except for the following entries:
\begin{align*}
\Expm(C_k/2)_{ii} &= \Exp(C_k/2)_{jj} = \Expm(C_k/2)_{i'i'} = \Exp(C_k/2)_{j'j'} = 0\\
\Expm(C_k/2)_{ii'} &= \Expm(C_k/2)_{jj'} = 1\\
\Expm(C_k/2)_{i'i} &= \Expm(C_k/2)_{j'j} = -1.
\end{align*}
It is then easily checked that $\Expm(C_k/2)^\top H_{i'j'} \Expm(C_k/2) = H_{ij}$.

\subsection{Numerical Study: Linear Optimization over $O_n$}

In this section, we conduct a numerical study comparing Riemannian gradient descent to tangent subspace descent for linear optimization over $O_n$:
\[ \min_{Y \in O_n} \left\{ f(Y):=\Tr(D^\top Y) \right\}. \]
This is equivalent to the well-known orthogonal Procrustes problem of minimizing a quadratic objective $\|A Y - B\|_F^2$ over $O_n$, which has numerous applications in science and engineering. Since $Y \in O_n$, it is easily seen that this is equivalent to maximizing $\Tr((A^\top B)^\top Y)$.

For Riemannian gradient descent, we use a backtracking line search rule to determine the step size at each iteration. For tangent subspace descent, we show that when $\CB_k$ are chosen to consist of a single pair $(i,j)$, an exact line search over $\eta > 0$ can be computed in closed form to optimize the objective of $Y^{k-1} \Expm(-\eta A^k)$. We use this rule to perform our updates in Algorithm \ref{alg:tangent-subspace-descent}.

\begin{proposition}\label{prop:stiefel-line-search}
Fix some $Y \in O_n$ and $H_{ij} \in \CB$. Let $G = D^\top Y$. Then
\[ \min_{\eta \geq 0} \Tr\left( D^\top Y \Expm(-\eta H_{ij}) \right) = \sum_{i' \neq i,j} g_{i'i'} - \sqrt{(g_{ii} + g_{jj})^2 + (g_{ij} - g_{ji})^2} \]
and the $(i,j)$-submatrix of $\Expm(-\eta^* H_{ij})$
is
\[
\begin{bmatrix}
\Expm(-\eta^* H_{ij})_{ii} & \Expm(-\eta^* H_{ij})_{ij}\\ \Expm(-\eta^* H_{ij})_{ji} & \Expm(-\eta^* H_{ij})_{jj}
\end{bmatrix}= -\frac{1}{\sqrt{(g_{ii} + g_{jj})^2 + (g_{ij} - g_{ji})^2}} \begin{bmatrix}
g_{ii} + g_{jj} & g_{ij} - g_{ji}\\ -(g_{ij} - g_{ji}) & g_{ii} + g_{jj}
\end{bmatrix}. 
\]
\end{proposition}

\begin{proof}[Proof of Proposition \ref{prop:stiefel-line-search}]
Observe that from \eqref{eq:orthongal-givens},
\[\Tr\left( D^\top Y \Expm(-\eta H_{ij}) \right) = \sum_{i' \neq i,j} g_{i'i'} + (g_{ii} + g_{jj}) \cos(-\eta) + (g_{ij} - g_{ji}) \sin(-\eta). \]
Since $(\cos(-\eta),\sin(-\eta))$ belong to the unit ball in $\bbR^2$, the result follows from linear optimization over the unit ball.\Halmos
\end{proof}

Test instances were generated synthetically, inspired by the orthogonal Procrustes problem. We generated $A \in \bbR^{n \times n}$ with i.i.d. $N(0,4)$ entries, a random orthogonal matrix $X \in O_n$, and $B = AX + \epsilon$ where $\epsilon \in \bbR^{n \times n}$ with i.i.d. $N(0,1)$ entries, then set $D = -A^\top B$. This is equivalent to minimizing $\|AY - B\|_F^2$. We tested $n = 50,100,150,200$. For each $n$, we generated 10 instances, and ran both algorithms. Experiments were run on a 2.5GHz personal machine with 16GB memory. We plotted the results in Figure \ref{fig:performance}. Generally, tangent subspace descent takes much fewer cycles to close the gap than Riemannian gradient descent.

\newcommand{\wid}{0.47\textwidth}
\newcommand{\hgt}{0.2}
\newcommand{\scalefactor}{0.85}
\pgfplotsset{footnotesize,
}
\begin{figure}[tb]
\centering
\begin{tikzpicture}
\begin{axis}[
name=ax1,
scale=\scalefactor,
width=\wid,height=\hgt\textheight,
title={$n=150$},
yticklabel style={/pgf/number format/fixed},
xlabel={\% cycles elapsed},
yticklabel style={/pgf/number format/fixed},
xmode=log,
log ticks with fixed point,
ylabel={\% gap closed},
legend style={font=\footnotesize, at={(1.5,-0.3)},anchor=north east,legend columns=2},
]
\addlegendentry{TSD}
\addplot [color=blue] table[x=percent, y=TSDcycle31, col sep=comma] {resn150.csv};
\addlegendentry{GD}
\addplot [color=red] table[x=percent, y=GDcycle31, col sep=comma] {resn150.csv};
\addplot [color=blue] table[x=percent, y=TSDcycle32, col sep=comma] {resn150.csv};
\addplot [color=red] table[x=percent, y=GDcycle32, col sep=comma] {resn150.csv};
\addplot [color=blue] table[x=percent, y=TSDcycle33, col sep=comma] {resn150.csv};
\addplot [color=red] table[x=percent, y=GDcycle33, col sep=comma] {resn150.csv};
\addplot [color=blue] table[x=percent, y=TSDcycle34, col sep=comma] {resn150.csv};
\addplot [color=red] table[x=percent, y=GDcycle34, col sep=comma] {resn150.csv};
\addplot [color=blue] table[x=percent, y=TSDcycle35, col sep=comma] {resn150.csv};
\addplot [color=red] table[x=percent, y=GDcycle35, col sep=comma] {resn150.csv};
\addplot [color=blue] table[x=percent, y=TSDcycle36, col sep=comma] {resn150.csv};
\addplot [color=red] table[x=percent, y=GDcycle36, col sep=comma] {resn150.csv};
\addplot [color=blue] table[x=percent, y=TSDcycle37, col sep=comma] {resn150.csv};
\addplot [color=red] table[x=percent, y=GDcycle37, col sep=comma] {resn150.csv};
\addplot [color=blue] table[x=percent, y=TSDcycle38, col sep=comma] {resn150.csv};
\addplot [color=red] table[x=percent, y=GDcycle38, col sep=comma] {resn150.csv};
\addplot [color=blue] table[x=percent, y=TSDcycle39, col sep=comma] {resn150.csv};
\addplot [color=red] table[x=percent, y=GDcycle39, col sep=comma] {resn150.csv};
\addplot [color=blue] table[x=percent, y=TSDcycle40, col sep=comma] {resn150.csv};
\addplot [color=red] table[x=percent, y=GDcycle40, col sep=comma] {resn150.csv};
\end{axis}

\begin{axis}[
at={(ax1.north west)},
yshift=1.5cm,
scale=\scalefactor,
width=\wid,height=\hgt\textheight,
title={$n=50$},
yticklabel style={/pgf/number format/fixed},
xlabel={\% cycles elapsed},
yticklabel style={/pgf/number format/fixed},
xmode=log,
log ticks with fixed point,
ylabel={\% gap closed},
]
\addplot [color=blue] table[x=percent, y=TSDcycle11, col sep=comma] {resn50.csv};
\addplot [color=red] table[x=percent, y=GDcycle11, col sep=comma] {resn50.csv};
\addplot [color=blue] table[x=percent, y=TSDcycle12, col sep=comma] {resn50.csv};
\addplot [color=red] table[x=percent, y=GDcycle12, col sep=comma] {resn50.csv};
\addplot [color=blue] table[x=percent, y=TSDcycle13, col sep=comma] {resn50.csv};
\addplot [color=red] table[x=percent, y=GDcycle13, col sep=comma] {resn50.csv};
\addplot [color=blue] table[x=percent, y=TSDcycle14, col sep=comma] {resn50.csv};
\addplot [color=red] table[x=percent, y=GDcycle14, col sep=comma] {resn50.csv};
\addplot [color=blue] table[x=percent, y=TSDcycle15, col sep=comma] {resn50.csv};
\addplot [color=red] table[x=percent, y=GDcycle15, col sep=comma] {resn50.csv};
\addplot [color=blue] table[x=percent, y=TSDcycle16, col sep=comma] {resn50.csv};
\addplot [color=red] table[x=percent, y=GDcycle16, col sep=comma] {resn50.csv};
\addplot [color=blue] table[x=percent, y=TSDcycle17, col sep=comma] {resn50.csv};
\addplot [color=red] table[x=percent, y=GDcycle17, col sep=comma] {resn50.csv};
\addplot [color=blue] table[x=percent, y=TSDcycle18, col sep=comma] {resn50.csv};
\addplot [color=red] table[x=percent, y=GDcycle18, col sep=comma] {resn50.csv};
\addplot [color=blue] table[x=percent, y=TSDcycle19, col sep=comma] {resn50.csv};
\addplot [color=red] table[x=percent, y=GDcycle19, col sep=comma] {resn50.csv};
\addplot [color=blue] table[x=percent, y=TSDcycle20, col sep=comma] {resn50.csv};
\addplot [color=red] table[x=percent, y=GDcycle20, col sep=comma] {resn50.csv};
\end{axis}

\begin{axis}[
at={(ax1.north east)},
xshift=2cm,
yshift=1.5cm,
scale=\scalefactor,
width=\wid,height=\hgt\textheight,
title={$n=100$},
yticklabel style={/pgf/number format/fixed},
xlabel={\% cycles elapsed},
yticklabel style={/pgf/number format/fixed},
xmode=log,
log ticks with fixed point,
ylabel={\% gap closed},
]
\addplot [color=blue] table[x=percent, y=TSDcycle21, col sep=comma] {resn100.csv};
\addplot [color=red] table[x=percent, y=GDcycle21, col sep=comma] {resn100.csv};
\addplot [color=blue] table[x=percent, y=TSDcycle22, col sep=comma] {resn100.csv};
\addplot [color=red] table[x=percent, y=GDcycle22, col sep=comma] {resn100.csv};
\addplot [color=blue] table[x=percent, y=TSDcycle23, col sep=comma] {resn100.csv};
\addplot [color=red] table[x=percent, y=GDcycle23, col sep=comma] {resn100.csv};
\addplot [color=blue] table[x=percent, y=TSDcycle24, col sep=comma] {resn100.csv};
\addplot [color=red] table[x=percent, y=GDcycle24, col sep=comma] {resn100.csv};
\addplot [color=blue] table[x=percent, y=TSDcycle25, col sep=comma] {resn100.csv};
\addplot [color=red] table[x=percent, y=GDcycle25, col sep=comma] {resn100.csv};
\addplot [color=blue] table[x=percent, y=TSDcycle26, col sep=comma] {resn100.csv};
\addplot [color=red] table[x=percent, y=GDcycle26, col sep=comma] {resn100.csv};
\addplot [color=blue] table[x=percent, y=TSDcycle27, col sep=comma] {resn100.csv};
\addplot [color=red] table[x=percent, y=GDcycle27, col sep=comma] {resn100.csv};
\addplot [color=blue] table[x=percent, y=TSDcycle28, col sep=comma] {resn100.csv};
\addplot [color=red] table[x=percent, y=GDcycle28, col sep=comma] {resn100.csv};
\addplot [color=blue] table[x=percent, y=TSDcycle29, col sep=comma] {resn100.csv};
\addplot [color=red] table[x=percent, y=GDcycle29, col sep=comma] {resn100.csv};
\addplot [color=blue] table[x=percent, y=TSDcycle30, col sep=comma] {resn100.csv};
\addplot [color=red] table[x=percent, y=GDcycle30, col sep=comma] {resn100.csv};
\end{axis}

\begin{axis}[
at={(ax1.south east)},
xshift=2cm,
scale=\scalefactor,
width=\wid,height=\hgt\textheight,
title={$n=200$},
yticklabel style={/pgf/number format/fixed},
xlabel={\% cycles elapsed},
yticklabel style={/pgf/number format/fixed},
xmode=log,
log ticks with fixed point,
ylabel={\% gap closed},
]
\addplot [color=blue] table[x=percent, y=TSDcycle41, col sep=comma] {resn200.csv};
\addplot [color=red] table[x=percent, y=GDcycle41, col sep=comma] {resn200.csv};
\addplot [color=blue] table[x=percent, y=TSDcycle42, col sep=comma] {resn200.csv};
\addplot [color=red] table[x=percent, y=GDcycle42, col sep=comma] {resn200.csv};
\addplot [color=blue] table[x=percent, y=TSDcycle43, col sep=comma] {resn200.csv};
\addplot [color=red] table[x=percent, y=GDcycle43, col sep=comma] {resn200.csv};
\addplot [color=blue] table[x=percent, y=TSDcycle44, col sep=comma] {resn200.csv};
\addplot [color=red] table[x=percent, y=GDcycle44, col sep=comma] {resn200.csv};
\addplot [color=blue] table[x=percent, y=TSDcycle45, col sep=comma] {resn200.csv};
\addplot [color=red] table[x=percent, y=GDcycle45, col sep=comma] {resn200.csv};
\addplot [color=blue] table[x=percent, y=TSDcycle46, col sep=comma] {resn200.csv};
\addplot [color=red] table[x=percent, y=GDcycle46, col sep=comma] {resn200.csv};
\addplot [color=blue] table[x=percent, y=TSDcycle47, col sep=comma] {resn200.csv};
\addplot [color=red] table[x=percent, y=GDcycle47, col sep=comma] {resn200.csv};
\addplot [color=blue] table[x=percent, y=TSDcycle48, col sep=comma] {resn200.csv};
\addplot [color=red] table[x=percent, y=GDcycle48, col sep=comma] {resn200.csv};
\addplot [color=blue] table[x=percent, y=TSDcycle49, col sep=comma] {resn200.csv};
\addplot [color=red] table[x=percent, y=GDcycle49, col sep=comma] {resn200.csv};
\addplot [color=blue] table[x=percent, y=TSDcycle50, col sep=comma] {resn200.csv};
\addplot [color=red] table[x=percent, y=GDcycle50, col sep=comma] {resn200.csv};
\end{axis}
\end{tikzpicture}
\caption{Results for TSD (blue) vs GD (red). The horizontal axis shows the cycles elapsed as a percentage of the largest number of cycles for that instance, while the vertical axis shows the gap closed, as a percentage of the best objective value found across both algorithms.}
\label{fig:performance}
\end{figure}
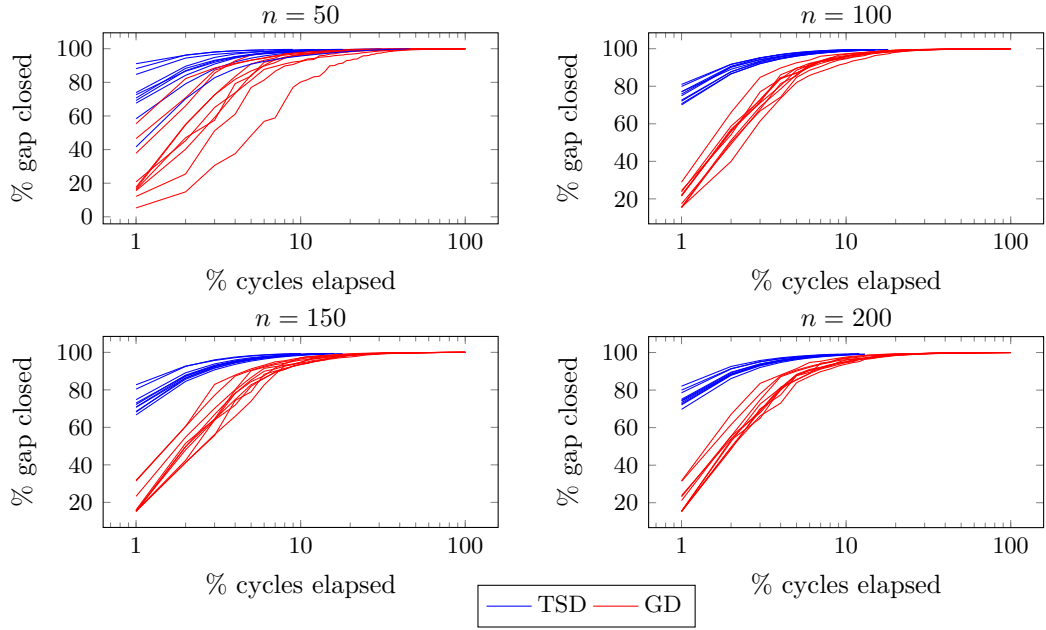

\section{Conclusion}\label{sec:conclusion}

In this paper we have presented and analysed a generalization of coordinate descent to manifold domains called tangent subspace descent. We found that the crux of ensuring convergence is a prudent choice of subspace selection rule $\cP$. For deterministic rules, are precisely those that satisfy the $(\gamma,r)$-gap ensuring condition while appropriate randomized rules satisfy the $C$-randomized norm condition. We proposed a new deterministic rule for optimization over the manifold of square orthogonal matrices, and applied TSD with this rule to the Orhthogonal Procrustes Problem. Our numerical experiments indicate the TSD's promise in application. We also provided a new randomized subspace selection rule for optimization over a general Stiefel manifold.

There are several interesting directions that can stem from this work, we give two here. First, it is known that accelerated rates for coordinate descent can be achieved with a simple modification; it is unknown whether something similar can be done by appealing to recent work \citep{ZhangSra2018} on extending Nesterov's accelerated gradient algorithm to manifolds. Second, it is well-known that replacing the exponential map with a retraction operator can result in computational savings; it would be interesting to see how retractions can be incorporated into the TSD framework.

\subsection*{Acknowledgements}
This research is, in part, funded by Award N660011824020 from the DARPA Lagrange Program and NSF Award 1740707.


\newpage

\appendix
\section{Auxiliary Results}\label{sec:appendix}

\begin{proposition}\label{prop:random-variable-convergence}
	Let $\{a_t,b_t\}_{t \in \bbN}$ be integrable sequences of random variables adapted to a filtration $\CF_1 \subseteq \CF_2 \subseteq \ldots \subseteq \CF$ of some probability space $(\Omega,\CF,\bbP)$, which almost surely satisfy for all $t \in \bbN$
	\begin{align*}
	0 \leq b_t &\leq a_t - \bbE[a_{t+1} \mid \CF_t]\\
	a_t &\geq a^* \in \bbR.
	\end{align*}
	Then $b_t \to 0$ almost surely.
\end{proposition}

\begin{proof}[Proof of Proposition \ref{prop:random-variable-convergence}]
	First, observe that by taking expectation, and using the law of total expectation, for each $t \in \bbN$, we have $\bbE[b_t] \leq \bbE[a_t] - \bbE[a_{t+1}]$. Since the right hand side telescopes, we have
	\[ \sum_{t \in T} \bbE[b_t] \leq \bbE[a_1] - \bbE[a_{T+1}] \leq \bbE[a_1] - a^*, \]
	where the second inequality follows since $a_{t+1} \geq a^*$ almost surely.
	
	Now fix some $\epsilon > 0$. By Markov's inequality, $\bbE[b_t]/\epsilon \geq \bbP[b_t \geq \epsilon]$, therefore substituting this into the previous inequality and taking $T \to \infty$ gives
	\[ \sum_{t=1}^\infty \bbP[b_t \geq \epsilon] \leq \frac{\bbE[a_1] - a^*}{\epsilon} < \infty. \]
	By the Borel-Cantelli lemma, we know that $\bbP[b_t \geq \epsilon \text{ infinitely often}] = 0$.
	
	In fact, this holds for any $\epsilon > 0$. Let $\{\epsilon_n\}_{n \in \bbN}$ be a sequence such that $\epsilon_n \to 0$. Then by the union bound
	\[ \bbP\left[ \exists n \text{ s.t. } b_t \geq \epsilon_n \text{ infinitely often} \right] \leq \sum_{n=1}^\infty \bbP[b_t \geq \epsilon_n \text{ infinitely often}] = 0. \]
	Finally, we observe that $\bbP[b_t \to 0] = 1-\bbP\left[ \exists n \text{ s.t. } b_t \geq \epsilon_n \text{ infinitely often} \right] = 1$.
\end{proof}

\begin{lemma}\label{lemma:projection}
	Let $V,W$ be two finite-dimensional vector spaces of equal dimension. Let $U:V \to W$ be an invertible unitary map (so $U U^{-1} = \id_W$, $U^{-1} U = \id_V$). Let $P:V \to V$ be a projection (i.e., $P = P^2$) onto a subspace $S$ of $V$. Then the projection onto the subspace $U S = \left\{ U v : v \in S \right\} \subset W$ is $U P U^{-1}$. 
\end{lemma}

\begin{proof}[Proof of Lemma \ref{lemma:projection}]
	Clearly, $(U P U^{-1}) (U P U^{-1}) = U P^2 U^{-1} = U P U^{-1}$ so it is a projection operator. It remains to check that the image is $US$. Clearly, we have
	$\Image(UPU^{-1}) \subseteq US$. Now take some $v \in S$. We want to show that there exists some $w \in W$ such that $U P U^{-1} w = U v$. We can easily take $w = U v$, which gives $U P U^{-1} w = U P U^{-1} (U v) = U P v = U v$. Thus, $\Image(UPU^{-1}) \supseteq US$.
\end{proof}

\begin{proposition}\label{prop:projection-norm}
	Let $U$,$U' \subseteq\R^n$ be subspaces of dimension $d$ in some Hilbert space with projection operators $P_U,P_{U'}$. If $u_1,\ldots,u_d$ and $u_1',\ldots,u_d'$ are orthonormal bases for $U$ and $U'$ respectively satisfying $\min_{i=1,\ldots,d}\ip{u_i}{u_i'}\geq\beta$ then $\|P_U-P_{U'}\|\leq d\sqrt{1-\beta^2}$.
\end{proposition}

\begin{proof}[Proof of Proposition \ref{prop:projection-norm}]
	Observe by orthonormality of the bases that
	\[
	P_U(\cdot) = \sum_{i=1}^d u_i \la u_i, \cdot \ra,\quad P_{U'}(\cdot) = \sum_{i=1}^d u'_i \la u'_i, \cdot \ra,
	\]
	so
	\[
	\|P_U-P_{U'}\|^2=\left\|\sum_{i=1}^d\left(u_i \la u_i, \cdot \ra - u_i' \la u_i', \cdot \ra \right)\right\|\leq\sum_{i=1}^d\left\| u_i \la u_i, \cdot \ra - u_i' \la u_i', \cdot \ra \right\|.
	\]
	It suffices to bound $\left\| u_i \la u_i, \cdot \ra - u_i' \la u_i', \cdot \ra \right\|$ for $1\leq i\leq d$. To simplify notation, let $u=u_i$ and $v=u'_i$. We can write $v=\alpha u+\sqrt{1-\alpha^2}w$ where $\alpha= \la u,v \ra$ and $w=\frac{v-\alpha u}{\sqrt{1-\alpha^2}}$ is a normal vector othogonal to $u$. Then we can write
	\[
	u \la u, \cdot \ra -v \la v,\cdot \ra =(1-\alpha^2) u \la u, \cdot \ra - \alpha \sqrt{1-\alpha^2} (u \la w, \cdot \ra + w \la u, \cdot \ra) - (1-\alpha^2)w \la w, \cdot \ra.
	\]
	Thus, $u \la u, \cdot \ra -v \la v,\cdot \ra$ essentially acts on the 2-dimensional subspace $\Span(\{u,v\})$ of the Hilbert space. A basis for this subspace is $\{u,w\}$, and the matrix representation of $u \la u, \cdot \ra -v \la v,\cdot \ra$ on this subspace is
	\[
	\begin{bmatrix}
	1-\alpha^2 & -\alpha\sqrt{1-\alpha^2}\\
	-\alpha\sqrt{1-\alpha^2} & -(1-\alpha^2)
	\end{bmatrix}.
	\]
	Hence, the eigenvalues are readily computed to be $\pm \sqrt{1-\alpha^2}$, which is bounded above by $\sqrt{1-\beta^2}$. Thus, $\| u \la u, \cdot \ra -v \la v,\cdot \ra \| \leq \sqrt{1-\beta^2}$, and
	\[
	\|P_U-P_{U'}\|^2\leq\sum_{i=1}^d\left\| u_i \la u_i, \cdot \ra - u_i' \la u_i', \cdot \ra \right\| \leq d \sqrt{1-\beta^2}.
	\]
\end{proof}

\end{document}